\documentclass[10pt]{amsart}
\usepackage{amscd}
\usepackage{amsmath}
\usepackage{amssymb}
\usepackage{latexsym}
\usepackage{lscape}
\usepackage{xypic}
\usepackage{hyperref} 
\usepackage{color}

\newtheorem{thm}{Theorem}[section]
\newtheorem{prop}[thm]{Proposition}
\newtheorem{lem}[thm]{Lemma}

\newtheorem{lem-def}[thm]{Lemma-Definition}
\newtheorem{cor}[thm]{Corollary}
\theoremstyle{remark}
\newtheorem{ex}[thm]{Example}

\newtheorem{rmk}{Remark}[section]
\theoremstyle{definition}
\newtheorem{dfn}{Definition}[section]
\newtheorem{conj}[thm]{Conjecture}
\newtheorem{defn}{Definition}[section]
\newtheorem{rem}{Remark}[section]

\numberwithin{equation}{section}

\newcommand{\quash}[1]{}  
\newcommand{\red}[1]{}  
\newcommand{\green}[1]{}  
\newcommand{\blue}[1]{}

\newcommand{\nc}{\newcommand}
\nc{\on}{\operatorname}

\newcommand{\frakg}{{\mathfrak g}}
\newcommand{\frakh}{{\mathfrak h}}

\newcommand{\frakm}{{\mathfrak m}}

\newcommand{\frakt}{{\mathfrak t}}

\newcommand{\bA}{{\mathbb A}}

\newcommand{\bC}{{\mathbb C}}
\newcommand{\bD}{{\mathbb D}}

\newcommand{\bG}{{\mathbb G}}

\newcommand{\bK}{{\mathbb K}}
\newcommand{\bL}{{\mathbb L}}

\newcommand{\bP}{{\mathbb P}}

\newcommand{\bV}{{\mathbb V}}
\newcommand{\bW}{{\mathbb W}}

\newcommand{\bZ}{{\mathbb Z}}

\newcommand{\calF}{{\mathcal F}}
\newcommand{\calG}{{\mathcal G}}

\newcommand{\calI}{{\mathcal I}}
\newcommand{\calJ}{{\mathcal J}}

\newcommand{\calL}{{\mathcal L}}

\newcommand{\calN}{{\mathcal N}}
\newcommand{\calO}{{\mathcal O}}

\newcommand{\calR}{{\mathcal R}}

\nc{\al}{{\alpha}} \nc{\be}{{\beta}} \nc{\ga}{{\gamma}}
\nc{\ve}{{\varepsilon}} \nc{\Ga}{{\Gamma}} \nc{\la}{{\lambda}}
\nc{\La}{{\Lambda}}

\nc{\ad}{{\on{ad}}}

\nc{\aff}{{\on{aff}}}
\nc{\Aff}{{\mathbf{Aff}}}
\newcommand{\Aut}{{\on{Aut}}}
\nc{\Bun}{{\on{Bun}}}

\nc{\der}{{\on{der}}}
\newcommand{\Der}{{\on{Der}}}
\nc{\diag}{{\on{diag}}}
\newcommand{\End}{{\on{End}}}
\nc{\Fl}{{\calF\ell}}

\newcommand{\Hom}{{\on{Hom}}}
\nc{\Hol}{{\on{Hol}}}
\newcommand{\id}{{\on{id}}}
\nc{\Id}{{\on{Id}}}

\nc{\Ind}{{\on{Ind}}}
\newcommand{\Lie}{{\on{Lie}}}
\newcommand{\Pic}{{\on{Pic}}}

\nc{\res}{{\on{res}}}

\newcommand{\Spec}{{\on{Spec}}}

\nc{\tr}{{\on{tr}}}

\newcommand{\Mod}{{\mathrm{-Mod}}}
\newcommand{\GL}{{\on{GL}}}

\nc{\GSp}{{\on{GSp}}} \nc{\GU}{{\on{GU}}} \nc{\SL}{{\on{SL}}}
\nc{\SU}{{\on{SU}}} \nc{\SO}{{\on{SO}}}

\nc{\four}{{\calF our}}

\def\xcoch{\mathbb{X}_\bullet}

\newcommand{\fg}{\mathfrak{g}}

\newcommand{\HH}{\mathbb{H}}

\newcommand{\CC}{\mathbb{C}}

\newcommand{\C}{\CC}
\newcommand{\PP}{\mathbb{P}}
\renewcommand{\P}{\PP}
\newcommand{\G}{\mathbb{G}}
\newcommand{\ZZ}{\mathbb Z}
\newcommand{\Z}{\ZZ}
\newcommand{\into}{\hookrightarrow}

\newcommand{\onto}{\twoheadrightarrow}

\newcommand{\ra}{\rightarrow}

\newcommand{\bra}{{\langle}}
\newcommand{\ket}{{\rangle}}

\newcommand{\Sym}{{\mbox {Sym~}}}
\newcommand{\noi}{\noindent}
\newcommand{\rk}{{\mbox{rk~}}}

\newcommand{\cO}{\mathcal{O}}

\newcommand{\cY}{\mathcal{Y}}

\def\question#1{{}}

\author{An Huang, Bong H. Lian, and Xinwen Zhu}

\title{Period Integrals and the Riemann-Hilbert Correspondence}

\begin{document}

\maketitle
\begin{abstract}
A tautological system, introduced in \cite{LSY}\cite{LY}, arises as a regular holonomic system of partial differential equations that governs the period integrals of a family of complete intersections in a complex manifold $X$, equipped with a suitable Lie group action. A geometric formula for the holonomic rank of such a system was conjectured in \cite{BHLSY}, and was verified for the case of projective homogeneous space under an assumption. In this paper, we prove this conjecture in full generality. By means of the Riemann-Hilbert correspondence and Fourier transforms, we also generalize the rank formula to an arbitrary projective manifold with a group action.
\end{abstract}

\tableofcontents
\baselineskip=16pt plus 1pt minus 1pt
\parskip=\baselineskip

\pagenumbering{arabic}
\addtocounter{page}{0}
\markboth{\SMALL
An Huang, Bong H. Lian, and Xinwen Zhu}
{\SMALL Period Integrals and the Riemann-Hilbert Correspondence}


\section{Introduction}\label{Intro}

Let $G$ be a connected algebraic group over a field $k$ of characteristic zero.
Let $X$ be a projective $G$-variety and let $\calL$ be a very
ample $G$-linearized invertible sheaf over $X$ which gives rise to a
$G$-equivariant embedding
\[X\to \bP(V),\]
where $V=\Gamma(X,\calL)^\vee$.
Let $r=\dim V$. We assume that the action of $G$ on $X$ is locally effective, i.e. $\ker(G\to\Aut(X))$ is finite. Let $\bG_m$ be the multiplicative group acting on $V$ by homotheties. Let $\hat{G}=G\times\bG_m$, whose Lie algebra is $\hat{\frakg}=\frakg\oplus ke$, where $e$ acts on $V$ by identity.
 We denote by $Z:\hat{G}\to \GL(V)$ the corresponding group representation, and $Z:\hat{\frakg}\to \End(V)$ the corresponding Lie algebra representation. 
Note that under our assumptions, $Z:\hat{\frakg}\to \End(V)$ is injective.

Let $\hat{\imath}:\hat{X}\subset V$ be the cone of $X$, defined by
the ideal $I(\hat{X})$. Let $\beta:\hat{\frakg}\to k$ be a Lie algebra
homomorphism. Then a {\it tautological system} as defined in \cite{LSY}\cite{LY} is the cyclic $D$-module on $V^\vee$
\[\tau(G,X,\calL,\beta)=D_{V^\vee}/D_{V^\vee}J(\hat{X})+D_{V^\vee}(Z(\xi)+\beta(\xi), \xi\in \hat{\frakg}),\]
where
$$J(\hat{X})=\{\widehat{D}\mid D\in I(\hat{X})\}$$
is the ideal of the commutative subalgebra $\bC[\partial]\subset
D_{V^\vee}$ obtained by the Fourier transform of $I(\hat{X})$ (see \S \ref{Appendix} for the review of the Fourier transform and in particular \eqref{appen:ring hom} for the notation).

Given a basis $\{a_i\}$ of $V$, we have $Z(\xi)=\sum_{ij}\xi_{ij}a_i\partial_{a_j}$, where $(\xi_{ij})$ is the matrix representing $\xi$ in the basis. Since the $a_i$ are also linear coordinates on $V^\vee$, we can view $Z(\xi)\in\Der k[V^\vee]\subset D_{V^\vee}$. In particular, the identity operator $Z(e)\in\End V$ becomes the Euler vector field on $V^\vee$. 

We recall the main motivation for studying tautological systems. Let $X'$ be a compact complex manifold (not necessarily algebraic), such that the complete linear system of anticanonical divisors in $X'$ is base point free.
Let $\pi:\cY\ra B:=\Gamma(X',\omega_{X'}^{-1})_{sm}$ be the family of smooth CY hyperplane sections $Y_a\subset X'$, and let $\HH^{top}$ be the Hodge bundle over $B$ whose fiber at $a\in B$ is the line $\Gamma(Y_a,\omega_{Y_a})\subset H^{n-1}(Y_a)$, where $n=\dim X'$. In \cite{LY}, the period integrals of this family are constructed by giving a canonical trivialization of $\HH^{top}$. Let $\Pi=\Pi(X')$ be the period sheaf of this family, i.e. the locally constant sheaf generated by the period integrals (Definition 1.1 \cite{LY}.)

Let $V=\Gamma(X',\omega_{X'}^{-1})^\vee$, $X$ be the image of the natural map $X'\to\P(V)$, and $\calL=\calO_X(1)$. Let $G$  be a connected algebraic group acting on $X$.

\begin{thm}\label{LY-theorem}
The period integrals of the family $\pi:\cY\ra B$ are solutions to
$$\tau=\tau(G,X,\calL,\beta_0),$$
where $\beta_0$ is the Lie algebra
homomorphism which vanishes on $\frakg$ and $\beta_0(e)=1$.
\end{thm}

This was proved in \cite{LSY} for $X'$ a partial flag variety, and in full generality in \cite{LY}, where the result was also generalized to hyperplane sections of general type. 

We note that when $X'$ is a projective homogeneous manifold of a semisimple group $G$, in which case we have $X=X'$, $\tau$ is amenable to explicit descriptions. For example, one description says that the tautological system can be generated by the vector fields corresponding to the linear $G$ action on $V^\vee$, and a twisted Euler vector field, together with a set of quadratic differential operators corresponding to the defining relations of $X$ in $\P(V)$ under the Pl\"ucker embedding. The case where $X$ is a Grassmannian has been worked out in detail \cite{LSY}. Furthermore, when the middle primitive cohomology $H^n(X)_{prim}=0$, it is also known that the system $\tau$ is complete, i.e. the solution sheaf coincides with the period sheaf \cite{BHLSY}. 

We now return to a general tautological system $\tau$. Applying an argument of \cite{Kapranov1997}, we find that
 if $G$ acts on $X$ by finitely many orbits, and if the character D-module on $\hat{G}$
 \[\calL_{\beta}:= D_{\hat{G}}/D_{\hat{G}}(\xi+\beta(\xi), \ \xi\in\hat{\frakg})\]
on $\hat{G}$ is regular singular,
 then $\tau$ is regular holonomic. See \cite{LSY} Theorem 3.4(1).
 In this case, if $X=\sqcup_{l=1}^rX_l$ is the decomposition into $G$-orbits, then the singular locus of $\tau$ is contained in $\cup_{l=1}^rX_l^\vee$. Here $X_l^\vee\subset V^\vee$ is the conical variety whose projectivization $\PP(X_l^\vee)$ is the projective dual to the Zariski closure of $X_l$ in $X$.
\quash{
 The regular singularity of $\tau$ can also be seen as follows. Let
\[\hat{\tau}:=\four(\tau)\]
be the Fourier transform of $\tau$.
Explicitly,
\[\hat{\tau}=D_V/D_VI(\hat{X})+D_V(Z^\vee(\xi)+ \beta'(\xi), \xi\in \frakg),\]
where $Z^\vee:\frakg\to \End(V^\vee)$ is the dual of $Z$ and
$\beta'=\beta-r\beta_0$
 By Lemma \ref{mono:eq}, $\hat{\tau}=\four(\tau)$ is $(\hat{G},\beta')$-equivariant. Clearly, $\hat{\tau}$ is set theoretically supported on $\hat{X}$.
 Applying Lemma \ref{mono:rh} and Lemma \ref{rh}, we obtain
\begin{cor} If $G$ acts on $X$ by finitely many orbits, then
$\hat{\tau}$ is holonomic. If $\calL_{-\beta'}$ is regular singular, then
$\hat{\tau}$ is regular singular. The same is true for $\tau$.
\end{cor}
}
From now on we assume that $G$ acts on $X$ by finitely many orbits, and $\calL_{\beta}$ is regular singular. Note that the latter assumption is always satisfied when $G$ is reductive.

Let us now turn to the main problem studied in this paper.
In the well-known applications of variation of Hodge structures in mirror symmetry, it is important to decide which solutions of our differential system come from period integrals. By Theorem \ref{LY-theorem}, the period sheaf is a subsheaf of the solution sheaf of a tautological system. Thus an important problem is to decide when the two sheaves actually coincide, i.e. when $\tau$ is complete. If $\tau$ is not complete, how much larger is the solution sheaf relative to the period sheaf? From Hodge theory, we know that  (see Proposition 6.3 \cite{BHLSY}) the rank of the period sheaf is given by the dimension of the middle vanishing cohomology of the smooth hypersurface $Y_a$.
Therefore, to answer those questions, it is clearly desirable to know precisely the holonomic rank of $\tau$. For a brief overview of known results on these questions in a number of special cases, see Introduction in \cite{BHLSY}.

\begin{conj} (Holonomic rank conjecture) \label{holo-rank}
Let $X$ be an $n$-dimensional projective homogeneous space of a semisimple group $G$. The solution rank of $\tau=\tau(G,X,\omega_X^{-1},\beta_0)$ at the point $a\in V^\vee$  is given by $\dim H_n(X-Y_a)$.
\end{conj}

In \cite{BHLSY}, the following is proved,
\begin{thm}\label{geo.thm}
Assume that the natural map
$$
\fg\otimes\Gamma(X,\omega_X^{-r})\ra\Gamma(X,T_X\otimes\omega_X^{-r})
$$
is surjective for each $r\geq0$. Then conjecture \ref{holo-rank} holds.
\end{thm}

In this paper, we prove this in full generality.

\begin{thm}\label{holo-rank-thm}
Conjecture \ref{holo-rank} holds.
\end{thm}

This will be proved in \S\ref{CY case}.
There are at least two immediate applications of this result. First we can now compute the solution rank for $\tau$ for {\it generic} $a\in V^\vee$.
\begin{cor}
The solution rank of $\tau$ at a smooth hyperplane section $a$ is
$$
\dim H^n(X)_{prim}+\dim H^{n-1}(Y_a)-\dim H^{n+1}(X).
$$
where the first term is the middle primitive cohomology of $X=G/P$ with $n=\dim X$.
\end{cor}
The last two terms of the rank above can be computed readily in terms of the semisimple group $G$ and the parabolic subgroup $P$ by the Lefschetz hyperplane and the Riemann-Roch theorems. (See Example 2.4 \cite{LSY}). A second application of Theorem \ref{holo-rank-thm} is to find certain exceptional points $a$ in $V^\vee$ where the solution sheaf of $\tau$ degenerates ``maximally''.

\begin{defn}\label{intro:def}
A nonzero section $a\in V^\vee=\Gamma(X,\omega_X^{-1})$ is called a rank 1 point if the solution rank of $\tau$ at $a$
 is 1. In other words, $\Hom_{D_{V^\vee}}(\tau,\cO_{V^\vee,a})\simeq\C$.
\end{defn}

\begin{cor}\label{rank1point}
Any projective homogeneous variety admits a rank 1 point.
\end{cor}

We will construct these rank 1 points in two explicit but different ways. The first, which works for $G=SL_l$, is a recursive procedure that produces such a rank 1 point by assembling rank 1 points from lower step flag varieties, starting from Grassmannians, and by repeatedly applying Theorem \ref{holo-rank-thm}. The second way, which works for any semisimple group $G$, is by using a well-known stratification of the flag variety $G/B$ to produce an open stratum in $X=G/P$ with a one dimensional middle degree cohomology. The complement of this stratum is an anticanonical divisor, hence a rank 1 point of $X$ by Theorem \ref{holo-rank-thm}.

The geometric formula in Conjecture \ref{holo-rank} appears to go well beyond the context of homogeneous spaces.
Theorem \ref{holo-rank-thm} can be seen as a special case of the following much more general theorem.
Consider a smooth projective $G$-variety $X$ with $\calL=\omega_X^{-1}$ very ample. Set $\tau=\tau(G,X,\calL,\beta)$ and $V^\vee=\Gamma(X,\calL)$. We introduce some more notations. Let $\bL^\vee$ be the total space of $\calL$ and $\mathring{\bL}^\vee$ be the complement of the zero section. Let $$ev:V^\vee\times X\onto\bL^\vee, \quad (a,x)\mapsto a(x)$$ be the evaluation map, and $\bL^\perp:=\ker(ev)$. Finally let
$$\pi^\vee: U:=V^\vee\times X-\bL^\perp\to V^\vee.$$
Note that this is the complement of the universal family of hyperplane sections $\bL^\perp\onto V^\vee, \ (a,x)\mapsto a$. Put $$D_{X,\beta}=(D_X\otimes k_\beta)\otimes_{U\hat\frakg} k,$$ where $k_\beta$ is the 1-dimensional $\hat{\frakg}$-module given by the character $\beta$ (see \S \ref{Appendix} for the notations). 
We now state a main result of this paper.
\begin{thm}\label{intro:main}
For $\beta(e)=1$, there is a canonical isomorphism
$$\tau\simeq H^0\pi^\vee_+(\calO_{V^\vee}\boxtimes D_{X,\beta})|_U.$$
\end{thm}

\begin{cor} \label{dim coho}
Suppose $G$ acts on $X$ by finitely many orbits,  and $k=\C$. Then the solution rank of $\tau$ at $a\in V^\vee$ is given by $\dim H_c^n(U_a, \on{Sol}(D_{X,\beta})|_{U_a})$, where $U_a=X-Y_a$.
\end{cor}

More generally,  we have

\begin{thm}\label{intro:general beta CY case}
For $\beta(e)\notin\bZ_{\leq0}$, and $\calL=\omega_X^{-1}$, there is a canonical isomorphism
$$\tau\simeq H^0\pi^\vee_+ev^!(D_{\mathring{\bL}^\vee,\beta})[1-r].$$
\end{thm}

In addition to proving Conjecture \ref{holo-rank} as a special case, Theorem \ref{intro:main} can also be used to derive the well-known formula for the solution rank of a GKZ system \cite{GKZ1990} at generic point $a$. 
But since Corollary \ref{dim coho} holds for arbitrary $a\in V^\vee$, it holds in particular for $a$ corresponding the union of all $T$-invariant divisors in $X$ (which is anticanonical). In this case, Theorem \ref{intro:main} implies that $a$ is a rank 1 point -- a result of \cite{HLY1996} based on Gr\"obner basis theory but motivated by applications to mirror symmetry. Thus Theorem \ref{intro:main} and Corollary \ref{dim coho} interpolate a result of \cite{GKZ1990} and \cite{HLY1996} by unifying the rank formula at generic point and at those exceptional rank 1 points, and at the same time, generalize them to an arbitrary $G$-variety.

Theorems \ref{intro:general beta CY case} and  \ref{special:main} are clearly motivated by period integral problems in Calabi-Yau geometry. Equally important parallel problems for manifolds of general type have also been systematically studied \cite{LY}\cite{CL2014}. In this paper,
we develop the general type analogues of those two main theorems. Roughly speaking, $\omega_X^{-1}$ is replaced by an arbitrary very ample invertible sheaf $\calL$ on $X$, and $\tau$ by a larger differential system defined on $\Gamma(X,\calL)^\vee\times\Gamma(X,\calL\otimes\omega_X)^\vee$. This class of systems arise naturally from period integrals of general type hypersurfaces in $X$. The precise statements will be formulated and proved in \S\ref{general type hyperplane} and \S\ref{general beta general type case}.

We now outline the paper. In \S\ref{CY case}, we prove Theorem \ref{intro:main} and a number of its consequences, including
Theorem \ref{holo-rank-thm}. We also describe explicitly the ``cycle-to-period'' map $H_n(X-Y_a)\ra\Hom_{D_{V^\vee}}(\tau,\cO_{V^\vee,a})$
as a result of Theorem \ref{intro:main}, and use it to answer a question recently communicated to us by S. Bloch.
While \S\ref{CY case} deals only with the case $\beta(e)=1$, we remove this assumption in \S\S\ref{!fiber}-\ref{exact seq}.
In \S\ref{!fiber}, we study the $!$-fibers of $\tau$, and describe some vanishing results at the special point $a=0$. We describe the geometric set up in \S\ref{geometry} for proving Theorem \ref{intro:general beta CY case}. The key step of the proof, involving an exact sequence for $\tau$, is done in \S\ref{exact seq}. In \S\ref{general type hyperplane} and \S\ref{general beta general type case}, we prove the general type analogues of Theorems \ref{intro:general beta CY case} and  \ref{special:main}.
 Finally, we apply our results to construct rank 1 points for partial flag varieties in the case $G=SL_l$ in \S\S\ref{partial1}-\ref{partial2}, and for general semisimple groups in \S\ref{G/P}. The appendix \S\ref{Appendix} collects some standard facts on D-modules.

\blue{ Should we mention sections 3 and 7 in the above paragraph? Or maybe we should not, as it is already mentioned before this paragraph?}

{\it Acknowledgements.} S. Bloch has independently noticed the essential role of the Riemann-Hilbert correspondence in connecting the de Rham cohomology and solution sheaf of a tautological system. We thank him for kindly sharing his observation with us. We also thank T. Lam for helpful communications. A.H. would like to thank S.-T. Yau for advice and continuing support, especially for providing valuable resources to facilitate his research. B.H.L. is partially supported by NSF FRG grant DMS 1159049. X.Z. is supported by NSF grant DMS-1313894 and DMS-1303296 and the AMS Centennial Fellowship.

\section{CY hyperplane sections}\label{CY case}

We begin with Theorem \ref{intro:main} : $X$ is a $G$-variety with $\calL=\omega_X^{-1}$ very ample, and $\beta(e)=1$. This is in fact a special case of the more general Theorem \ref{intro:general beta CY case} and therefore can be also obtained by the methods introduced in later sections. However, we decide to deal with this case first for several reasons. On the one hand, the proof given here is different from the later method and is more direct. On the other hand, the subcase when $\beta(\frakg)=0$, i.e. $\beta=\beta_0$, which is important to mirror symmetry, is already covered by Theorem \ref{intro:main}.

Let $n=\dim X$. Let $U=V^\vee\times X-V(f)$, where $V(f)=\bL^\perp$ is the universal hyperplane section, so that $U_a=X-V(f_a)$ where $V(f_a)=Y_a$, the zero locus of the section $f_a\equiv a\in V^\vee$. Let $\pi^\vee:U\to V^\vee$ denote the projection. The restriction of $\beta$ to $\frakg$ is still denoted by $\beta$ when no confusion arises.
Put $D_{X,\beta}=(D_X\otimes k_\beta)\otimes_{U\frakg} k$. Note that if $G$ acts on $X$ by finitely many orbits, then $D_{X,\beta}$ is $(G,\beta)$-equivariant holonomic D-module on $X$ (see Lemma \ref{mono:rh} and \ref{mono:eq}), and therefore 
$$\calN:=(\calO_{V^\vee}\boxtimes D_{X,\beta})|_U$$ 
is a holonomic D-module on $U$.

\begin{thm}\label{special:main}
Assume that $\beta(e)=1$. Then there is a canonical isomorphism $\tau\simeq H^0\pi^\vee_+\calN$.
\end{thm}

\begin{cor}\label{surjective}
 If $\beta({\frakg})=0$. There is a canonical surjective map
\[\tau\to H^0\pi^\vee_+\calO_U.\]
\end{cor}
\begin{proof}
Note that there is always a surjective map $D_{X,0}=D_X/D_X\frakg\to D_X/D_XT_X=\calO_X$. The corollary follows from the fact that $\pi^\vee_+$ is right exact as
$\pi^\vee:  U\to V^\vee$ is affine.
\end{proof}

We turn to the solution sheaf of $\tau$ via the Riemann-Hilbert correspondence.
Assume $G$ acts on $X$ by finitely many orbits. Let us write $\calF=\on{Sol}(D_{X,\beta})$. This is a perverse sheaf on $X$.
\begin{cor}\label{special:rank formula}
Let $k=\C$ and $a\in V^\vee$. Then the solution rank of $\tau$ at $a$ is given by $\dim H_c^{0}(U_a, \calF|_{U_a})$.
\end{cor}
\begin{proof}Denote $\on{Sol}(\calN)=\calG$.
According to the Riemann-Hilbert correspondence, $\on{Sol}(\tau)={^p}R^0\pi^\vee_!\calG$, where ${^p}R^0\pi^\vee_!$ denotes the $0$th perverse cohomology of $\pi^\vee_!$. Then the non-derived solution sheaf ${^{cl}}\on{Sol}(\tau)=\Hom_{D_{V^\vee}}(\tau,\calO_{V^\vee})$ is given by $H^{-r}({^p}R^0\pi^\vee_!\calG)$, the $(-r)$th (standard) sheaf cohomology of  ${^p}R^0\pi^\vee_!\calG$. However, as $R\pi^\vee_!\calG$ lives in positive perverse degrees, $H^{-r}({^p}R^0\pi^\vee_!\calG)=H^{-r}R\pi^\vee_!\calG= R^{-r}\pi^\vee_!\calG$. As $\calG=\bC[r]\boxtimes {\calF} |_U$, the claim follows.
\end{proof}

\begin{rmk}
We will give more explicit descriptions of the perverse sheaf $\calF$ in various situations later on. For example, in the case $X$ is a homogenous $G$-variety and $\beta({\frakg})=0$, then $\calF=\bC[n]$.
\end{rmk}

Now we prove Theorem \ref{special:main}. We will assume $\beta({\frakg})=0$ to simplify notations.

Let us write
\begin{equation}\label{special:R}
\calR:=D_{V^\vee}/D_{V^\vee}J(\hat{X}),
\end{equation}
which is a left $D_{V^\vee}$-module. Observe that
for any $\xi\in\hat{\frakg}$,
$D_VI(\hat{X})Z^\vee(\xi)\subset D_VI(\hat{X})$,
so
$D_{V^\vee}J(\hat{X})Z(\xi)\subset D_{V^\vee}J(\hat{X})$.
Therefore, $D_{V^\vee}J(\hat{X})$ can be regarded as a \emph{right}
$\hat{\frakg}$-module, on which $\xi\in\hat{\frakg}$ acts via the right
multiplication by $Z(\xi)$.
Accordingly, $\calR$ is also a right
$\hat{\frakg}$-module. In addition, by definition we have
\begin{equation}\label{special:exp of tau}
\tau=(\calR\otimes k_\beta)\otimes_{\hat{\frakg}}k,
\end{equation}
where $k_\beta$ is the 1-dimensional representation of $\hat{\frakg}$ defined by $\beta$.

We now convert $\calR$ to a left $\hat\frakg$-module (cf. \cite[\S 2]{BHLSY}.) Let $\{a_i\}$ be a basis of $V$ and $\{a_i^*\}$ the dual basis.
Observe that as $\calO_{V^\vee}$-modules, one can write
\[\calR\simeq \calO_{V^\vee}\otimes S,\]
where
\begin{equation}\label{speical:S}
S=k[\partial_{a_i}]/J(\hat{X})\simeq \calO_V/I(\hat{X})
\end{equation}
is
identified with the homogeneous coordinate ring of $\hat{X}$, and
$\calO_{V^\vee}$ acts on the first factor\footnote{The
$D_{V^\vee}$-module structure on $\calR$ is given as follows:
$\partial_{a_i}$ acts on $\calO_{V^\vee}\otimes S$ as $\partial_{a_i}\otimes
1+1\otimes a_i^*$.}. If we convert the right action of $\hat{\frakg}$ on
$\calR$ described above to a left action $\alpha$, then $\alpha$ will be the sum of the following two
actions: the first is the action of $\hat{\frakg}$ on the second factor
through the dual representation $Z^\vee:\hat{\frakg}\to \End V^\vee\to \End S$, which is denoted by
$\alpha_1$; to describe the second action $\alpha_2$, observe that
the natural multiplication map
$$(V\otimes V^\vee)\otimes (\calO_{V^\vee}\otimes S)\to
(\calO_{V^\vee}\otimes S),$$ induces $V\otimes V^\vee\to\End\calR$
and $\alpha_2$ is via $Z^\vee:\hat{\frakg}\to V\otimes V^\vee\to \End (\calR)$.
Explicitly, if we write $a\otimes b\in \calO_{V^\vee}\otimes S$,
then
\begin{equation}\label{special:al1}
\alpha_1(\xi)(a\otimes b)=a\otimes Z^\vee(\xi)(b).
\end{equation}
Let's write
$Z^\vee(\xi)=-\sum_{ij} \xi_{ij}a_i\otimes a_j^*$. Then
\begin{equation}\label{special:al2}
\alpha_2(\xi)(a\otimes b)=- \sum_{ij}\xi_{ij} aa_i\otimes ba_j^*.
\end{equation}

Let $f=\sum a_i\otimes a_i^*\in \calR$, which can be regarded as the universal section of the line bundle $\calO_{V^\vee}\boxtimes\calL$ over $V^\vee\times X$. Recall that $U=V^\vee\times X-\bL^\perp$. Then
\[\calO_U=(\calO_{V^\vee}\otimes S(\hat{X}))_{(f)}\]
is the homogeneous localization of $\calR$ with respect to $f$, where the degree of $a\otimes b\in\calO_{V^\vee}\otimes S(\hat{X})$ is the degree of $b$ in the graded ring $S(\hat{X})$. As $\calL^{-1}=\omega_X$, we can regard $f^{-1}$ as a rational section of $\calO_{V^\vee}\boxtimes\omega_X$, regular on $U$. Then $\calO_Uf^{-1}$ can be identified with the regular sections of $\calO_{V^\vee}\boxtimes\omega_X$ over $U$. In other words,
\begin{equation}\label{id}
\calO_Uf^{-1}\simeq\omega_{U/V^\vee}=(\calO_{V^\vee}\boxtimes\omega_X)|_U.
\end{equation}
Therefore, it is equipped with a $(D_{V^\vee}\boxtimes D_X^{op})|_U$ module structure (see \cite[VI, \S 3]{Borel} for the definition of right $D_X$ module structure on $\omega_X$). As $\frakg$ maps to the vector fields on $X$, $\calO_Uf^{-1}$ is a $D_{V^\vee}\times\frakg$-module.  We will describe this structure more explicitly. First, we describe the $D_{V^\vee}$-module structure. Let $\theta$ be a vector field on $V^\vee$, and $\xi\in\frakg$. It is enough to describe $\theta(f^{-1})$ and $(f^{-1})\xi$. Let us write $Z^\vee(\xi)=-\sum_{ij} \xi_{ij}a_i\otimes a_j^*$ as before. 
\begin{lem}\label{special:formula}
We have
\[\theta(f^{-1})=- \frac{\sum_i \theta(a_i)\otimes a_i^*}{f^2}\in \calO_Uf^{-1},\]
and
\[(f^{-1})\xi=-\frac{\sum_{ij} \xi_{ij}a_i\otimes a_j^*}{f^2}\in \calO_Uf^{-1}.\]
\end{lem}
\begin{proof}
Let $v\in V^\vee$, regarded as a section of $\calL$. Then $v^{-1}$ is a rational section of $\omega_X$, and $\omega=1\otimes v^{-1}$ is a rational section of $\calO_{V^\vee}\boxtimes \omega_X$, obtained by pullback of a rational section of $\omega_X$. Note that $g=(1\otimes v)/f\in \calO_U$, and we can write $f^{-1}=g (1\otimes v^{-1})$. By definition, for a vector field $\theta$ on $V^\vee$, $\theta(\omega)=0$, and for $\xi\in\frakg$, $\omega\xi=-1\otimes \Lie_\xi v^{-1}$, where $\Lie_\xi:\omega_X\to \omega_X$ is the Lie derivative (see \cite[VI, \S 3]{Borel} for the definition of right D-module structures on $\omega_X$). Therefore
\[\theta(f^{-1})=\theta(g)\omega, \quad (f^{-1})\xi= (g\xi)\omega-g(1\otimes \Lie_{\xi}v^{-1}).\]
Note that
\[\theta(g)=\theta(\frac{1}{\sum a_i\otimes \frac{a_i^*}{v}})=-\frac{\sum \theta(a_i)\otimes\frac{a_i^*}{v}}{(\sum a_i\otimes \frac{a_i^*}{v})^2}=
-g^2\sum \theta(a_i)\otimes\frac{a_i^*}{v}=-g\frac{\sum\theta(a_i)\otimes a_i^*}{f}.\]
Therefore, the first equation holds. On the other hand
\[(g)\xi=(\frac{1\otimes v}{f})\xi=-\frac{1\otimes Z^\vee(\xi)(v)}{f}- \frac{(1\otimes v)\sum \xi_{ij}a_i\otimes a^*_j}{f^2}.\]
To prove the second, we need to understand $\Lie_\xi v^{-1}$. We consider a more general situation.

Let $X$ be a Fano variety. Assume that $\calL=\omega_X^{-1}$ is very ample, and $X\to \bP(V)$ be the closed embedding where $V=\Gamma(X,\calL)^\vee$. Then $\frakg=\Gamma(X, T_X)$ is a Lie algebra and $\calL$ is naturally $\frakg$-linearized. Therefore, $V^\vee=\Gamma(X,\calL)$ is a natural $\frakg$-module with action $Z^\vee:\frakg\to \End(V^\vee)$. As $Z^\vee(\xi)=-\sum_{ij} \xi_{ij} a_i\otimes a_j^*$, we have $Z^\vee(\xi)(v)=-\sum_{ij} \xi_{ij}a_i(v)a_j^*$ for $v\in V^\vee$. On the other hand, recall that $\frakg$ acts on $\omega_X$ by Lie derivatives. Note that for $v\in V^\vee$, $v^{-1}$ can be regarded as a rational section of $\omega_X$. 

\begin{lem}\label{Lie der}
Let $\xi\in\frakg, 0\neq v\in V^\vee$. Then
\[\Lie_\xi v^{-1}=- \frac{Z^\vee(\xi)(v)}{v} v^{-1}.\]
\end{lem}
\begin{proof}
Consider the 1-parameter subgroup $g_t=exp(t\xi)$. Then
\[
{d\over dt}g_t^*(v^{-1})=-(v\circ g_t)^{-2}{d\over dt}(v\circ g_t).
\]
Now set $t=0$.
\end{proof}
Now Lemma \ref{special:formula} follows.
\end{proof}

Note that explicitly, the $D_{V^\vee}\times \frakg$-module structure on $\calO_Uf^{-1}$ can be described as follows. Let $\theta=\partial_{a_i}$ be a vector field on $V^\vee$ and $\xi=-\sum \xi_{ij}a_i\otimes a_j^*\in\frakg$,
$m=\frac{1}{f^{l+1}}(a\otimes b)\in \calO_Uf^{-1}$, where
$a\in\calO_{V^\vee}$ and $b\in S$ is homogeneous of degree $k$, then
$$\partial_{a_i}(m)=\frac{\partial_{a_i}(a)\otimes b}{f^{l+1}}+(-1)^{l+1}(l+1)\frac{a\otimes ba^*_i}{f^{l+2}},$$
$$(m)\xi=\frac{1}{f^{l+1}}(a\otimes Z^\vee(\xi)(b))-\frac{l+1}{f^{l+2}}(\sum_{ij} \xi_{ij} aa_i\otimes ba_j^*).$$
We extend this to a $\hat{\frakg}$-module by requiring that $e$ acts by zero on $\calO_Uf^{-1}$.

Now, we have the following technical lemma. Recall that $\beta(e)=1$.
\begin{lem}\label{special:isom}
The map  $\phi:\calR\otimes k_\beta\to \calO_Uf^{-1}$ given by
\[\phi(a\otimes b)=\frac{(-1)^l l! }{f^{l+1}}a\otimes b\]
is a $D_{V^\vee}\times \hat{\frakg}$-module homomorphism. In addition, it induces an isomorphism
\[\tau= (\calR\otimes \beta)\otimes_{\hat{\frakg}}k\simeq (\calO_Uf^{-1})\otimes_{\hat{\frakg}}k=  (\calO_Uf^{-1})\otimes_{\frakg}k.\]
\end{lem}
\begin{proof}
A direct calculation shows that $\phi$ is a $D_{V^\vee}\times\hat{\frakg}$-module homomorphism. Namely, we know that $\partial_{a_i}$ acts on $\calR$ by $\partial_{a_i}\otimes 1+1\otimes a^*_i$. Therefore,
\[\phi(\partial_{a_i}(a\otimes b))=\phi(\partial_{a_i}(a)\otimes b+a\otimes ba_i^*)=\frac{(-1)^l l!}{f^{l+1}}(\partial_{a_i}(a)\otimes b)+\frac{(-1)^{l+1}(l+1)!}{f^{l+2}}(a\otimes ba_i^*),\]
which is the same as $\partial_{a_i}\phi(a\otimes b)$. The $\frakg$-equivariance can be checked similarly.

Clearly $\phi$ is surjective, with the kernel spanned by
$(l+1)a\otimes b+ f(a\otimes b)$ for $b$ homogeneous of degree $l$.
But $(a\otimes b)\alpha(e)=(l+1)a\otimes b+ f(a\otimes b)$. The
lemma is proved.
\end{proof}

To apply this lemma, recall the definition of $\pi^\vee_+$ for $\pi^\vee: U\to V^\vee$ a smooth morphism of algebraic varieties. As $\pi^\vee$ is an affine morphism,
\[\pi^\vee_+\calN=\Omega_{U/V^\vee}^\bullet\otimes \calN[\dim X].\]
In particular,
\[H^0\pi^\vee_+\calN=\on{coker}((\calO_{V^\vee}\boxtimes\Omega_X^{\dim X-1}\otimes D_X\otimes_{\frakg} k)|_U\to  (\calO_{V^\vee}\boxtimes \omega_X\otimes D_X\otimes_{\frakg} k)|_U).\]
As $\on{coker}(\Omega_X^{\dim X-1}\otimes D_X\to \omega_X\otimes D_X)=\omega_X$ as right $D_X$-modules, $H^0\pi^\vee_+\calN$ is exactly $(\calO_{V^\vee}\boxtimes\omega_X)|_U\otimes_{\frakg} k\simeq\tau$. This completes the proof of Theorem \ref{special:main}.

We continue to let $X$ be a general smooth projective $G$-variety, and let $\beta(e)=1$. We further assume that $k=\bC$ and $\beta(\frakg)=0$, and consider some consequences of Theorem \ref{special:main}.
By taking the solution sheaves on both sides in Corollary \ref{surjective}, we get an injective map
\begin{equation}\label{special:cycl-to-period}
H_n(X-V(f_a))\simeq \Hom(H^0\pi^\vee_+\calO_U,\calO_{V^\vee,a})\to \Hom(\tau,\calO_{V^\vee,a}),
\end{equation}
where the first isomorphism follows from the same argument as in Corollary \ref{special:rank formula} and the Poincare duality.
This gives an explicit lower bound for the solution rank of $\tau$ at any point $a$. For applications, we need to give a more geometric and explicit description of this map.

Note that we can interpret $1/f$ as a family (parametrized by $V^\vee$) of meromorphic top forms on $X$, whose fiber over $a\in V^\vee$ has poles along $V(f_a)$. We denote this family of top forms on $X$ by $\Omega_a$. These forms can also be given as follows.

Consider the principal $\bG_m$-bundle $\pi^\vee:\mathring{\bL}^\vee\to X$ (with right action). Then there is a natural one-to-one correspondence between sections of $\calL$ and $\bG_m$-equivariant morphism $f:\mathring{\bL}^\vee\ra k$, i.e. $f(m\cdot h^{-1})=h f(m)$.  We shall write $f_a$ the function that represents the section $a$. Let $w=(w_1,..,w_n)$ be local coordinates on $X$, and $z_w$ be the coordinate induced on the fibers of $\bL^\vee$. Put $\omega=dz_w\wedge dw_1\wedge\cdots\wedge  dw_n$. Then it can be shown that $\omega$ defines a global non-vanishing form on $\bL^\vee$. (See \cite[Prop. 6.1]{LY}.) Let $x_0$ be the vector field generated by $1\in k=Lie(\bG_m)$. Then $\Omega:=i_{x_0}\omega$ is a $G$-invariant $\bG_m$-horizontal form of degree $\dim X$ on $\mathring{\bL}^\vee$. Moreover, since
\[\Omega_a:={\Omega\over f_a}\]
is $G\times\G_m$-invariant, it defines a family of meromorphic top form on $X$ with pole along $V(f_a)$ \cite[Thm. 6.3]{LY}.
Then the isomorphism in Lemma \ref{special:isom} sends the generator ``$1$" of $\tau$ to $\Omega_a$. Consider the ``cycle-to-period'' map defined in \cite{LY}
\[H_n(X-V(f_a))\to \Hom(\tau,\calO_{V^\vee,a}),\quad \gamma\mapsto \int_\gamma\Omega_a.\]

\begin{cor}\label{cycle-to-period}
The cycle-to-period map
$H_n(X-V(f_a))\ra \Hom(\tau,\calO_{V^\vee,a})$,
$$\gamma\mapsto\bra\gamma,{\Omega\over f_a}\ket=\int_\gamma{\Omega\over f_a},$$
is injective.
\end{cor}
\quash{
\begin{proof}
For $\gamma\in\Hom(H^0\pi_+^\vee\cO_U,\calO_{V^\vee,a})=H_n(X-V(f_a))$, its image is a $D_{V^\vee}$-linear function $\tau\simeq\omega_{U/V^\vee}\otimes_{\frakg}k
\to\calO_{V^\vee,a}$. Since $\tau_a$ is generated by ${\Omega\over f_a}$, if $\bra\gamma,{\Omega\over f_a}\ket=0$ then $\bra\gamma,\tau_a\ket=0$, hence $\gamma=0$.
\end{proof}
}

The rest of the section will not be used in the sequel. We note that the argument of Corollary \ref{special:rank formula} has the following interesting topological consequence, which answers a question S. Bloch communicated to us. Let $X\subset \bP^N$ be an $n$-dimensional smooth projective variety. Let $V(f)\to V^\vee=\Gamma(X,\calL)$ be the universal family of hyperplane sections of $X$.

\begin{cor}
Let $a\in V^\vee$. Then for $a'$ close to $a$, the map $H_n(X-V(f_a))\to H_n(X-V(f_{a'}))$ induced by parallel transport is injective.
\end{cor}

\begin{proof}
As argued in Corollary \ref{special:rank formula}, $H_n(X-V(f_a))$ can be identified with the stalk of the classical solutions of some regular holonomic system on $V^\vee$. Since any analytic solution to a regular holonomic system at $a$ extends to some neighborhood of $a$, the map between stalks of the classical solution sheaf of this regular holonomic system given by analytic continuation is injective.
\end{proof}

Our result sheds new light on the well-studied toric case, i.e. the original GKZ A-hypergeometric differential equations.
We assume that $X$ is a toric variety, with the action of the torus $G=T$. Then $\hat G=T\times \bG_m$. Then Theorem \ref{special:main} takes a particular easy form in the following situation.
\begin{cor}\cite{HLY1996}
If $\beta=\beta_0$, and $X$ is smooth toric variety, $G=T$ is the algebraic torus of $X$, and $Y_a$ is the anticanonical divisor of $X$ given by the union of $G$-invariant toric divisors in $X$, then $a$ is a rank 1 point.
\end{cor}
\begin{proof}Note that in this case $D_{X,\beta}|_{X-Y_a}\simeq \calO_{X-Y_a}$. Therefore, $\Hom(\tau,\calO_{V^\vee,a})\simeq H_n(T^n)$, which is one-dimensional.
\end{proof}

\section{$!$-fibers of $\tau$}\label{!fiber}

In the following three sections, we consider $\tau$ when $\beta(e)$ is not necessarily 1.
Here we will give a formula of the $!$-fibers of $\tau$
at $a\in V^\vee$. For $a\in V^\vee$, let $i_a:\{a\}\to V^\vee$ be
the inclusion and for simplicity, let us write
\[\tau_a^!=i_a^!\tau.\]
This is a complex of vector spaces and our goal is to give an
expression of this complex.

By \eqref{special:exp of tau} we have
\[\tau_a^!=k_a\otimes^L_{\calO_{V^\vee}}((\calR \otimes \beta)\otimes_{\hat{\frakg}}k)[-\dim V],\]
where $k_a=\calO_{V^\vee}/\frakm_a$ is the residual field at $a$,
and $\frakm_a$ is the maximal ideal of $\calO_{V^\vee}$
corresponding to $a$.

The advantage of this expression of $\tau_a^!$ is that we can first
calculate $k_a\otimes^L_{\calO_{V^\vee}}\calR$ as a (complex of)
right $\hat{\frakg}$-modules, and then taking the Lie algebra
coinvariants. Namely, we have the Koszul resolution of $k_a$, which gives the complex that calculates $\tau_a^!$
\begin{equation}\label{shrik stalk}
\tau_a^!=(\bigwedge V\otimes \calO_{V^\vee}\otimes
S)\otimes_{\hat{\frakg}}(-\beta).
\end{equation}
where $V\otimes \calO_{V^\vee}\to \calO_{V^\vee}$ is given by $v\otimes 1\mapsto v-v(a)$.
In general, this complex is difficult to compute. However, when $a=0$, this is more tractable, as we shall see.

First, for a general  point $a\in V^\vee$ we can express the degree $r$-term as
\begin{equation}\label{shrik r}
H^r\tau_a^!\simeq H_0(\hat{\frakg},S\otimes\beta),
\end{equation}
where the action of $\hat{\frakg}$ on $S$ will be the sum of two actions (induced by the actions $\alpha_1$ and $\alpha_2$ of $\frakg$ on $\calO_{V^\vee}\otimes S$, as described in \eqref{special:al1} and \eqref{special:al2}). Concretely, the first action is  via $Z^\vee:\hat{\frakg}\to \End V^\vee\to \End S$, and the second is via the $\xi(b)=-\sum \xi_{ij}a_i(a)ba_j^*$ for $b\in S$. If $a\neq 0$, $S$ is not a finite dimensional $\hat{\frakg}$-module and this Lie algebra coinvariant is difficult to compute. On the other hand, if $a=0$, the second action vanishes and $S$ decomposes as finite dimensional representations of $\hat{\frakg}$.

\begin{lem}\label{vanishing r}
Assume that $\beta(e)\not\in\bZ_{\leq 0}$. Then $H^r\tau_0^!=0$.
\end{lem}
\begin{proof}
The homothety $\bG_m$ acts on $S$ by nonnegative weights.
Therefore, if  $\beta(e)\not\in\bZ_{\leq 0}$, the coinvariant of $S\otimes\beta$ with respect to this $\bG_m$ is zero.
\end{proof}

From now on, we assume that $\beta(e)\not\in\bZ_{\leq 0}$.

Let us calculate $H^{r-1}\tau_0^!$. We have
\[\begin{CD}
(V\wedge V)\otimes (\calR\otimes \beta)@>m_2>>V\otimes (\calR\otimes \beta)@>m_1>>\calR\otimes \beta\\
@VVV@VVV@VVV\\
(V\wedge V)\otimes (\calR\otimes \beta)\otimes_{\hat{\frakg}}k@>d_2>> V\otimes (\calR\otimes \beta)\otimes_{\hat{\frakg}}k@>d_1>> (\calR\otimes \beta)\otimes_{\hat{\frakg}}k.
\end{CD}\]
Then
\[H^{r-1}\tau_0^!= m_1^{-1}((\calR\otimes\beta)\hat{\frakg})/( \on{Im}m_2+V\otimes(\calR\otimes\beta)\hat{\frakg})\]
As the Koszul complex is acyclic away from degree zero, we can rewrite the above as
\[H^{r-1}\tau_0^!=(\calR\otimes \beta)\hat{\frakg}\cap\on{Im}m_1/(\on{Im}m_1)\hat{\frakg}.\]
Consider
\[0\to (\calR\otimes \beta)\hat{\frakg}\cap\on{Im}m_1/(\on{Im}m_1)\hat{\frakg}\to (\calR\otimes\beta)\hat{\frakg}/(\on{Im}m_1)\hat{\frakg}\to (\calR\otimes\beta)\hat{\frakg}/ (\calR\otimes \beta)\hat{\frakg}\cap\on{Im}m_1\to 0. \]

Note that $0=H^r\tau_0^!$ implies that $(\calR\otimes\beta)\hat{\frakg}+\on{Im}m_1=\calR\otimes\beta$. Therefore,
\[(\calR\otimes\beta)\hat{\frakg}/(\calR\otimes \beta)\hat{\frakg}\cap\on{Im}m_1= \calR\otimes\beta/\on{Im}m_1.\]
We therefore can write
\[H^{r-1}\tau_0^!=\ker((\calR\otimes \beta)\hat{\frakg}/(\on{Im}m_1)\hat{\frakg}\to \calR\otimes\beta/\on{Im}m_1).\]
Therefore, there is a surjective map
\[H_1(\hat{\frakg},S\otimes\beta)\to H^{r-1}\tau_0^!,\]
where $\hat{\frakg}$ acts on $S$ via $Z$. (So $S$ are direct sums of finite dimensional representations of $\hat{\frakg}$.)

\begin{lem}\label{d vanishing}
For $\beta(e)\not\in\bZ_{\leq 0}$, we have $H_1(\hat{\frakg},S\otimes\beta)=0$. Therefore, $H^{r-1}\tau_0^!=0$.
\end{lem}
\begin{proof}
Consider the $\hat{\frakg}$ coinvariants functor as the composition of $\frakg$ coinvariants functor, and the $\bC$ coinvariants functor. The $E_2$ terms of the Grothendieck spectral sequence contributing to $H_1(\hat{\frakg},S\otimes\beta)$ are $H_1(\bC,H_0(\frakg,S\otimes\beta))$ and $H_0(\bC,H_1(\frakg,S\otimes\beta))$. As $S\otimes\beta$ breaks as direct sums according to weights as a $\frakg$-module, and $\bC$ acts on each given weight piece as the weight plus $\beta(e)$, it is clear that under the above assumption on $\beta(e)$, both $H_1(\bC,H_0(\frakg,S\otimes\beta))$ and $H_0(\bC,H_1(\frakg,S\otimes\beta))$ are zero.
\end{proof}

\section{The geometry}\label{geometry}
Let $X$ be a smooth projective variety and $\calL$ a very ample line
bundle which gives $X\to \bP(V)$, where $V^\vee=\Gamma(X,\calL)$.
Let $\hat{\imath}:\hat{X}\to V$ be the closed embedding of the cone of
$X$ into $V$. Let $\bL$ be the totally space of $\calL^\vee$. Then
$$i_\bL:\bL\to X\times V$$
is a rank one subbundle of the trivial vector bundle over $X$ with
fiber $V$. The following diagram is commutative
\[\begin{CD}
\bL@>i_\bL>>X\times V\\
@V\pi VV@VV\pi V\\
\hat{X}@>\hat{\imath}>>V
\end{CD}\]
and the left vertical arrow realizes $\bL$ as the blow-up of
$\hat{X}$ at the origin. We denote the open immersion
$$
j_{\mathring{\bL}}: \mathring{\bL}=\bL-X\to \bL,
$$
where $X$ is regarded as the zero section of $\bL$.

Let $\bL^\vee$ be the dual of $\bL$, i.e., the total space of $\calL$, and $j_{\mathring{\bL}^\vee}: \mathring{\bL}^\vee\to\bL^\vee$ be the open subset away from the zero section. The the dual of $i_\bL$ is the evaluation map
\[ev: X\times V^\vee\to \bL^\vee\]
which sends $(x,a)$ to $a(x)\in \bL^\vee$.

Let $i_{\bL^\perp}:\bL^\perp\to X\times V^\vee$ be the orthogonal
complement of $\bL$ in $X\times V^\vee$, i.e. the kernel of $ev$. The projection
\[\bL^\perp\stackrel{i_{\bL^\perp}}{\to} X\times V^\vee\stackrel{\pi^\vee}{\to} V^\vee\]
realizes $\bL^\perp$ as the universal family of hyperplane sections
of $X$. We still denote this projection by $\pi^\vee$. Let
$j_U:U=X\times V^\vee-\bL^\perp\to X\times V^\vee$ be the
complement. For $a\in V^\vee$, the fiber $U_a$ of $U\to V^\vee$ over
$a$ is $X-V(f_a)$, where $f_a$ is the section of $\calL$ given by
$a$ and $V(f_a)$ is its divisor. Note that the following diagram is Cartesian.
\begin{equation}\label{geom:cart}
\begin{CD}
U@>j_U>>X\times V^\vee\\
@V ev VV @VV ev V\\
\mathring{\bL}^\vee@>j_{\mathring{\bL}^\vee}>> \bL^\vee.
\end{CD}
\end{equation}

\quash{

Then the family version of Example
\ref{mix} is

\begin{lem}\label{obv}
$$\four_X((i_{\bL}j_{\mathring{\bL}})_+\calO_{\mathring{\bL}})\simeq
(j_U)_!\calO_U, \ \ \ \
\four_X((i_{\bL}j_{\mathring{\bL}})_!\calO_{\mathring{\bL}})\simeq
(j_U)_+\calO_U.$$
\end{lem}
}

\section{A formula for $\tau$: CY case}\label{exact seq}

We will complete the proof of Theorem \ref{intro:general beta CY case} in this section.

Let $i_0:\{0\}\to V$ be the inclusion of the origin, and
$j_0:\mathring{V}\to V$ be the open embedding of the complement. Let
$\mathring{X}=\hat{X}-\{0\}$. The open inclusion $\mathring{X}\to
\hat{X}$ is still denoted by $j_0$ and the closed inclusion
$\mathring{X}\to \mathring{V}$ is denoted by $\mathring{\imath}$.
By specializing \eqref{triangleII}, we have the following important sequence for $\hat\tau=\four(\tau)$
\begin{equation}\label{4-term}
0\to i_{0,+} H^{-1}i_0^+\hat{\tau}\to  H^0j_{0,!}(\hat{\tau}|_{\mathring{V}})\to
\hat{\tau}\to i_{0,+} H^0i_0^+\hat{\tau}\to 0.
\end{equation}

First we make a simplification of this sequence.

\begin{lem}\label{vanishing l}
For $\beta(e)\not\in\bZ_{\leq 0}$,
$i_{0,+} H^0i_0^+\hat{\tau}=0$.
\end{lem}
\begin{proof}Assume that $H^0i_0^+\hat{\tau}=k^\ell$, so that $i_{0,+} H^0i_0^+\hat{\tau}=\delta_0^\ell$. I.e. there is a surjective map of D-modules $\hat{\tau}\to \delta_0^\ell$ on $V$. Taking the Fourier transform, we therefore have a surjective map $\tau\to \calO_{V^\vee}^\ell$. Taking the right exact functor $H^ri_0^!$, i.e., the $r$th cohomology of the $!$-fibers at $0\in V^\vee$, we have a surjective map $H^r\tau^!_0\to k^\ell$. By Lemma \ref{vanishing r}, $\ell=0$.
\end{proof}
As a result, under our assumption
\begin{equation}\label{3-term}
0\to  i_{0,+} H^{-1}i_0^+\hat{\tau}\to  H^0j_{0,!}(\hat{\tau}|_{\mathring{V}})\to
\hat{\tau}\to 0.
\end{equation}
Let $d=\dim_k H^{-1}i_0^+\hat{\tau}$.
Then $i_{0,+}H^{-1}i_0^+\hat{\tau}=\delta_0^d$. Taking the Fourier transform of this sequence, we therefore obtain
\begin{equation}
0\to \calO_{V^\vee}^d\to \four(H^0j_{0,!}(\hat{\tau}|_{\mathring{V}}))\to \tau \to 0.
\end{equation}

We next understand $\four(H^0j_{0,!}(\hat{\tau}|_{\mathring{V}}))$. Clearly, $\hat{\tau}$ is set-theoretically supported on $\hat{X}$. Note that the Fourier transform of $Z(\xi)+\beta(\xi)$ is $Z^\vee(\xi)+\beta'(\xi)$, where
$$\beta'(\xi)=\beta(\xi)-\tr Z(\xi).$$ 
We have the following lemma.
\begin{lem} \label{restriction} 
For $\calL=\omega_X^{-1}$, we have
$\hat{\tau}|_{\mathring{X}}
 = D_{\mathring{X},\beta'}$, where $D_{\mathring{X},\beta'}$ is the D-module on $\mathring{X}$ as introduced in Lemma \ref{mono:eq}.
\end{lem}
\begin{proof}
Recall that $\hat\tau$ is defined as
\[\hat\tau=D_V/ D_VI(\hat X)+D_V (Z^\vee(\xi)+\beta'(\xi), \xi \in\hat\frakg).\]
Let $\calR':= D_V/ D_V I(\hat X)$. Then similar to \eqref{special:exp of tau}, $\hat\tau=(\calR'\otimes k_{\beta'})\otimes_{\hat\frakg} k$.
Consider the closed embedding $\mathring{X}\to\mathring{V}$. Then as explained in \S \ref{Appendix}, there is a $D_{\mathring{V}}\times D_{\mathring{X}}$-bimodule,
$D_{\mathring{V}\leftarrow \mathring{X}}=D_{\mathring{V}}/D_{\mathring{V}}\calI(\mathring{X})\otimes \omega_{\mathring{X}/\mathring{V}}$. As the relative canonical sheaf $\omega_{\mathring{X}/\mathring{V}}$ is trivial, $\calR'|_{\mathring{V}}=D_{\mathring{V}\leftarrow \mathring{X}}$. In particular, $\calR'$ also admits a right $D_{\mathring{X}}$ structure, and it is clear that the right action of $\hat{\frakg}$ on $\calR'$ is induced from the map $\hat{\frakg}\to D_{\mathring{X}}$.
Therefore, $\calR'|_{\mathring{X}}=D_{\mathring{X}}$ as $D_{\mathring{X}}$-bimodules, and $\hat{\tau}|_{\mathring{X}}=D_{\mathring{X},\beta'}$, as claimed.
\end{proof}

Note that $\mathring{\bL}\simeq \mathring{X}$ and therefore $\hat{\tau}|_{\mathring{X}}\simeq D_{\mathring{X},\beta'}$ can be regarded as a D-module on $\mathring{\bL}$, which is naturally $(\bG_m,\beta'(e))$-equivariant. Then
\begin{equation}\label{aux1}
H^0j_{0,!}(\hat{\tau}|_{\mathring{V}})=H^0\pi_!(i_{\bL,!}j_{\mathring{\bL},!}D_{\mathring{X},\beta'}).
\end{equation}
According to Lemma \ref{key},
\begin{equation}\label{aux2}
\four(H^0\pi_!(i_{\bL,!}j_{\mathring{\bL},!}D_{\mathring{X},\beta'}))=H^0\pi^\vee_!\four_X(i_{\bL,!}j_{\mathring{\bL},!}D_{\mathring{X},\beta'}).
\end{equation}
By \eqref{fcs},
\begin{equation}\label{aux3}
\four_X(i_{\bL,!}j_{\mathring{\bL},!}D_{\mathring{X},\beta'})=ev^!\four_X(j_{\mathring{\bL},!}D_{\mathring{X},\beta'})[1-r].
\end{equation}

\begin{lem}\label{aux4}
There is a canonical isomorphism
\[\four_X(j_{\mathring{\bL},!}D_{\mathring{X},\beta'})\simeq j_{\mathring{\bL}^{\vee},+}D_{\mathring{\bL}^\vee,\beta}.\]
\end{lem}
\begin{proof}
Instead of the original formula, we can prove $\four_X(j_{\mathring{\bL}^{\vee},+}D_{\mathring{\bL}^\vee,\beta})\simeq j_{\mathring{\bL},!}D_{\mathring{X},\beta'}$. First note that the $+$-restriction of
$\four_X(j_{\mathring{\bL}^{\vee},+}D_{\mathring{\bL}^\vee,\beta})$ along $X\to \bL$ is zero. In fact, we have the following more general fact. We keep the notations $\bL, \bL^\vee, j_{\mathring{\bL}}$ etc. We consider the $\bG_m$-action on $\bL$ by homothethies.

\begin{lem} Let $M$ be a $\bG_m$-monodromic holonomic D-module on $\mathring{\bL}$ (see \S \ref{Appendix} for the terminology.) Then the $+$-fiber of $\four_X(j_{\mathring{\bL},!}M)$ along $X\to \bL^\vee$ is zero.
\end{lem}
\begin{proof}One can check this pointwise on $X$ and by the base change of Fourier transform (Lemma \ref{key}), one can assume $X$ is a point. Then it follows from Example \ref{four:mono}.
\end{proof}

Therefore, by \eqref{triangleII}, it is enough to show $\four_X(j_{\mathring{\bL}^{\vee},+}D_{\mathring{\bL}^\vee,\beta})|_{\mathring{\bL}}=D_{\mathring{\bL},\beta'}$.
By definition, we can write
\[\four_X(j_{\mathring{\bL}^{\vee},+}D_{\mathring{\bL}^\vee,\beta})|_{\mathring{\bL}}= p_{\bL,+}(e^x\otimes D_{\mathring{\bL}\times_X\mathring{\bL}^\vee\to\mathring{\bL}^\vee}\otimes_{\frakg}(k_{-\beta}))[1].\]
In other words, $\four_X(j_{\mathring{\bL}^{\vee},+}D_{\mathring{\bL}^\vee,\beta})|_{\mathring{\bL}}$ is calculated as the cokernel of the map
\[e^x\otimes D_{\mathring{\bL}\times_X\mathring{\bL}^\vee\to\mathring{\bL}^\vee}\otimes_{\frakg}(k_{-\beta})\stackrel{\nabla}{\to}\Omega_{\mathring{\bL}^\vee/X}\otimes_{\calO_{\mathring{\bL}^\vee}} e^x\otimes D_{\mathring{\bL}\times_X\mathring{\bL}^\vee\to\mathring{\bL}^\vee}\otimes_{\frakg}(k_{-\beta}).\]
Note that $M= D_{\mathring{\bL}\times_X\mathring{\bL}^\vee\to\mathring{\bL}^\vee}\otimes_{\frakg}(k_{-\beta})$ is a cyclic $D_{\mathring{\bL}\times_X\mathring{\bL}^\vee}$-module, with a canonical generator ``$1$". For a local section $D\in D_{\mathring{\bL}\times_X\mathring{\bL}^\vee}$, let $[D]=D``1"$ denote the corresponding local section of $M$.
Note that $\Omega_{\mathring{\bL}/X}$ and $e^x$ are canonically trivialized as $\calO$-modules. Indeed, by definition, the underlying $\calO$-module of $e^x$ is the structure sheaf. On the other hand, if locally on $X$, we choose $s$ a section of $\calL$, regarded as a coordinate function on $\bL$,  and $t$ the dual coordinate on $\bL^\vee$. Then the 1-form $dt/t$ is independent of the choice and defines the trivialization of $\Omega_{\mathring{\bL}^\vee/X}$.
Therefore,  the underlying $\calO$-modules of  both terms in this complex are $M$. Then the D-module structure is given as follows: for $D\in D_{\mathring{\bL}\times_X\mathring{\bL}^\vee}$,
\[t\partial_t([D])=[t\partial_t D]+ [ts D], \quad s\partial_s([D])=[s\partial_s D]+[tsD].\]
As a result, $\four_X(j_{\mathring{\bL}^{\vee},+}D_{\mathring{\bL}^\vee,\beta})|_{\mathring{\bL}}=M/(t\partial_t+ts)M$.

\blue{ I tried to follow the arguments deriving the identity above, and I guess some more details may be helpful to the reader:}

More explicitly, as $\calO$-modules, $e^x:=m^!e^x$ is canonically trivialized, as was said above. Let $f\in\Gamma(\calO_{V\times V^\vee})$, then unravelling the definitions, we have the following action of $\partial_t$ on the element $f\otimes m^{-1}(1)\in e^x$:

\begin{equation}
\partial_t(f\otimes m^{-1}(1))=\partial_tf\otimes m^{-1}(1)+ f\partial_t(st)\otimes m^{-1}(\partial_{st}1)=(\partial_t+s)f\otimes m^{-1}(1)
\end{equation}
Note that $\partial_{st}=1$ in $e^x$. Therefore, for $1\otimes 1\otimes [D]\in e^x\otimes M$, we have
\begin{equation}
\nabla(1\otimes 1\otimes [D])=dt\otimes \partial_t(1\otimes [D])=dt\otimes s\otimes [D]+dt\otimes 1\otimes [\partial_tD]=\frac{dt}{t}\otimes 1\otimes [tsD+t\partial_tD]
\end{equation}
So one gets the above identity for the Fourier transform.

To proceed, we first consider $N= D_{\mathring{\bL}\times_X\mathring{\bL}^\vee\to\mathring{\bL}^\vee}/(t\partial_t+ts)D_{\mathring{\bL}\times_X\mathring{\bL}^\vee\to\mathring{\bL}^\vee}$, which is a D-module on $\mathring{\bL}$. We define a D-module homomorphism $D_{\mathring{\bL}}\to N, \quad D\mapsto D``1"$, which we claim is an isomorphism. Indeed, we can assume that $X$ is affine and the line bundle $\bL\to X$ is trivial. Then it is a direct calculation.

Finally,  note that both $N$ and $D_{\mathring{\bL}}$ are right $\hat{\frakg}$-modules. The $\hat{\frakg}$-module structure on $N$ comes from $\hat{\frakg}\to D_{\mathring{\bL}^\vee}$ acting on $D_{\mathring{\bL}\times_X\mathring{\bL}^\vee\to\mathring{\bL}^\vee}$ from the right, and the $\hat{\frakg}$-module structure on $D_{\mathring{\bL}}$ comes from $\hat{\frakg}\to D_{\mathring{\bL}^\vee}$ acting itself from the right. Under the above isomorphism, $D_{\mathring{\bL}}=N\otimes k_{\beta'-\beta}$. Now Lemma \ref{aux4} follows.
\end{proof}

\begin{lem}\label{for d} Assume that $\beta(e)\not\in\bZ_{\leq 0}$.
We have $d=\dim H^{r-1}\tau^!_0=0$.
\end{lem}
\begin{proof}
The second statement follows from Lemma \ref{d vanishing}. We need to establish the first equality.
For simplicity, let us denote $\calN:=ev^!(D_{\mathring{\bL}^\vee,-\beta})[1-r]$. This is a plain D-module on $U$.

Taking $i_0^!$ of \eqref{3-term}, it is enough to show that
$$H^ri^!_0H^0\pi^\vee_+\calN=H^{r-1}i_0^!H^0\pi^\vee_+\calN=0.$$
Consider the distinguished triangle
\[i_0^!H^{\leq -1}\pi^\vee_+\calN\to i_0^!\pi^\vee_+\calN\to i_0^!H^{0}\pi^\vee_+\calN\to.\]
The long exact sequence associated to this triangle is
\[H^{r-1}i_0^!\pi^\vee_+\calN\to H^{r-1}i_0^!H^{0}\pi^\vee_+\calN\to H^ri_0^!H^{\leq -1}\pi^\vee_+\calN\to H^ri_0^!\pi^\vee_+\calN\to H^ri_0^!H^{0}\pi^\vee_+\calN\to 0.\]
Note that $U$ does not intersect with $X\times\{0\}\subset X\times V^\vee$. Therefore, $i^!_0\pi^\vee_+\calN=0$. This implies that $H^ri^!_0H^0\pi^\vee_+\calN=0$, and $H^{r-1}i^!_0H^0\pi^\vee_+\calN=H^ri_0^!H^{\leq -1}\pi^\vee_+\calN$. But $H^{\leq -1}\pi^\vee_+\calN$ sits in cohomological degree $\leq -1$ and $i^!_0$ has cohomological amplitude $r$, $H^ri_0^!H^{\leq -1}\pi^\vee_+\calN=0$.
\end{proof}

We can now complete the proof of Theorem \ref{intro:general beta CY case}.
\begin{proof}
Combining \eqref{aux1}-\eqref{aux3} and Lemma \ref{aux4}, we can rewrite \eqref{3-term} as
\begin{equation}\label{3-term,2nd}
0\to \calO_{V^\vee}^d\to H^0\pi^\vee_+ev^!(D_{\mathring{\bL}^\vee,\beta})[1-r]\to \tau\to 0.
\end{equation}
Theorem \ref{intro:general beta CY case} follows immediately from Lemma \ref{for d} and the sequence \eqref{3-term,2nd}.
\end{proof}

\begin{rmk}
Note that explicitly,
$$ev^!(D_{\mathring{\bL}^\vee,\beta})[1-r]=D_{U}/ D_{U}T_{U/\mathring{\bL}^\vee}+ D_{U}(\xi+\beta(\xi), \xi \in\hat{\frakg}),$$
where $T_{U/\mathring{\bL}^\vee}$ is the relative tangent sheaf, and $\hat{\frakg}$ acts on $X\times V^\vee$ diagonally. In the special case $\beta(e)=1$, it reduces to $\calN=(\calO_{V^\vee}\boxtimes D_{X,\beta})|_U$ as in Theorem \ref{special:main}.
\end{rmk}

\section{General type hyperplane sections}\label{general type hyperplane}

Let $X$ be a projective $G$-variety, $\calL$ a very
ample $G$-linearized invertible sheaf over $X$, and
\[X\to \bP(V)\]
the associated $G$-equivariant embedding, where $V=\Gamma(X,\calL)^\vee$. Put $W=\Gamma(X,\calL\otimes\omega_X)^\vee$, $r=\dim V$, and $s=\dim W$. 

For simplicity, we assume that $\calL\otimes\omega_X$ is base point free. (That $W\neq0$ actually suffices for the following results.) Thus, we have a morphism $X\to \bP(V)\times\bP(W)$. Let 
$$\calI\subset k[V\times W]$$ 
be the bihomogeneous ideal defining the image, and let $\calI_d$ be the subspace of $\calI$ consisting of the $\deg_W=d$ elements.

Let $\bG_m^2$ be the multiplicative group acting on $V\times W$ by homotheties. Let $\hat{G}=G\times\bG_m^2$, whose Lie algebra is $\hat{\frakg}=\frakg\oplus ke^V\oplus ke^W$, where $e^V,e^W$ act respectively on $V,W$ by their identities.
 We denote by $Z^V:\hat{G}\to \GL(V)$ and $Z^W:\hat{G}\to\GL(W)$ the corresponding group representations, and $Z^V:\hat{\frakg}\to \End(V)$, $Z^W:\hat{\frakg}\to \End(W)$ the corresponding Lie algebra representations. In particular, $Z^V(e^V),Z^W(e^W)$ are the respective Euler vector fields on $V,W$. As before, we denote the Fourier transform by $~~\widehat{}:D_{V^\vee\times W^\vee}\ra D_{V\times W}$.

Let $\hat{\imath}:\hat{X}\subset V$ be the cone of $X$, defined by
the ideal $I(\hat{X})$. Let $\beta:\hat{\frakg}\to k$ be a Lie algebra
homomorphism. We extend the definition of a {\it tautological system} given in \S\ref{Intro} as follows \cite{LY}.
\begin{dfn}\label{enhanced tau}
Let $\tau_{VW}=\tau_{VW}(G,X,\calL,\beta)$ be the cyclic $D$-module  on $V^\vee\times W^\vee$ given by
\begin{align*}
D_{V^\vee\times W^\vee}/D_{V^\vee\times W^\vee}\calJ+D_{V^\vee\times W^\vee}J^W
+D_{V^\vee\times W^\vee}(Z^V(\xi)+Z^W(\xi)+\beta(\xi), \xi\in \hat{\frakg})
\end{align*}
where
$$\calJ=\widehat{\calI},\hskip.2in
J^W=\widehat{\Sym^2 W^\vee}.
$$
\end{dfn}

Note that when $\beta(e^W)=0$, we have $\tau_{VW}=\tau\boxtimes\cO_{W^\vee}$ where $\tau=\tau(G,X,\calL,\beta)$ is as defined in \S\ref{Intro}. (See first paragraph of \S\ref{general beta general type}.)


To apply Definition \ref{enhanced tau} to the geometric problem at hand, we first prove

\begin{prop}\label{general type classical solutions}
Let $\Pi$ be the sheaf generated by the period integrals $\Pi_\gamma$ of the universal family of hyperplane sections for $\calL$. Then we have an injective map $\Pi\ra {^{cl}}\on{Sol}(\tau_{VW})$, with $\beta(\frakg)=0$, $\beta(e^V)=1$ and $\beta(e^W)=-1$.
\end{prop} 


\begin{proof}
By construction \cite{LY}, $\Pi_\gamma=\int_\gamma {f_b\Omega\over f_a}$, where $b\in W^\vee$, $a\in V^\vee$, and $\Omega$ is a $G$-invariant $\G_m^2$-horizontal form of degree $\dim X$ on $\mathring{\bL}^\vee\oplus\mathring{\bK}^\vee$ \cite{CL2014}. Here $\bL^\vee, \bK^\vee$ are the respective total spaces of $\calL$, $\omega_X$. Note that the ${f_b\Omega\over f_a}$ define a family of meromorphic forms on $X$. Observe that $\calI_0$ is nothing but $I(\hat X)\subset k[V]$, the defining ideal of $X$ in $\P(V)$. Thus by \cite[Theorem 8.9]{LY}, $\Pi_\gamma$ is annihilated by the Fourier transform $\widehat{\calI_0}$. Since $\Pi_\gamma$ is linear along the component $W^\vee$, the period integral is automatically annihilated by $\widehat{\calI_d}$ for any $d>1$. Likewise, $J^W\Pi_\gamma=0$. As shown in \cite[\S8]{LY}, for a given homogeneous function $p\in\calI_1$, we have $\hat p {f_b\Omega\over f_a}=(-1)^l p {f_b\Omega\over f_a}$ where $l=\deg_V p$. But since $p\in\calI$, this form vanishes when it is restricted to $X$. It follows that $\widehat{\calI_1}\Pi_\gamma=0$.
Finally, by \cite[Theorem 8.9]{LY} again
$$(Z^V(\xi)+Z^W(\xi)+\beta(\xi))\Pi_\gamma=\Pi_\gamma,\hskip.2in\xi\in\frakg\oplus k e^V$$
where $\beta(\frakg)=0$ and $\beta(e^V)=1$. But since $\Pi_\gamma$ is linear along $W^\vee$, this condition is equivalent to
$$
(Z^V(\xi)+Z^W(\xi)+\beta(\xi))\Pi_\gamma=0, \hskip.2in \xi\in\hat\frakg\equiv\frakg\oplus k e^V\oplus k e^W
$$
with $\beta(e^W)=-1$. Therefore, the period integrals $\Pi_\gamma$ are analytic solutions to the differential systems associated to $\tau_{VW}(G,X,\calL,\beta)$, as desired.
\end{proof}

Returning to the general case of $\tau_{VW}\equiv\tau_{VW}(G,X,\calL,\beta)$,  we proceed to analyze it in a way parallel to \S\ref{CY case}. We shall follow most of the notations introduced there, but with a general line bundle $\calL$ now playing the role of $\omega_X^{-1}$ there. We will spell out the changes that need to be made to incorporate new structures associated to $W^\vee$ and $\hat\frakg=\frakg\oplus ke^V\oplus ke^W$. Put
$$
\calN:=(\cO_{V^\vee\times W^\vee}\boxtimes D_{X,\beta})|_{U\times W^\vee}.
$$
The following is a generalization of Theorem \ref{special:main}.

\begin{thm}\label{general:main} 
Assume that $\beta(e^V)=1$ and $\beta(e^W)=-1$. Then
there is a canonical isomorphism $\tau_{VW}\simeq H^0(\pi^\vee\times\id_{W^\vee})_+\calN$.
\end{thm}


For simplicity, we assume that $\beta(\frakg)=0$. The key step of the proof is finding an appropriate analogue of Lemma \ref{special:isom}, which we now formulate. Put
$$
\calR^V=D_{V^\vee}/D_{V^\vee}\calI_0,\hskip.5in
\calR^W=D_{W^\vee}/D_{W^\vee}J^W.
$$
Then $\calR^V,\calR^W$ have {\it right} $\hat\frakg$-module structures by right multiplications via $Z^V,Z^W$ respectively. Put
$$\calR^{VW}=R/R \widehat{\calI_1},\hskip.2in R:=\calR^V\boxtimes\calR^W.$$ 
For the same reason as $\widehat{\calI_1}$ also affords an action of $G$, $R/R \widehat{\calI_1}$ has a {\it right} $\hat\frakg$-module structures by right multiplications. (Note that $\widehat{\calI_1}$ a priori lives in a bigger space, whereas we used the same notation to denote its image in the quotient $R$.) By definition we have
$$
\tau_{VW}=(\calR^{VW}\otimes k_{\beta})\otimes_{\hat\frakg}k.
$$

Fix bases $a_1,..,a_r$ of $V$, and $b_1,..,b_s$ of $W$ respectively. As in \S\ref{CY case}, we have as $\cO_{V^\vee}$-modules
$$
\calR^V\simeq\cO_{V^\vee}\otimes S^V
$$
where $S^V=\cO_V/\calI_0$, which is $\bZ_{\geq0}$-graded. The $D_{V^\vee}$-structure on $\calR^V$ is then given by $\partial_{a_i}\mapsto\partial_{a_i}\otimes 1+1\otimes a_i^*$. We can also convert the right $\hat\frakg$-action on $\calR^V$ to a left action $\alpha$ as before. Similarly, we have as $\cO_{W^\vee}$-modules
$$
\calR^W\simeq\cO_{W^\vee}\otimes S^W
$$
where $S^W=\cO_W/\cO_W\Sym^2 W^\vee$, which is $\bZ/2\bZ$-graded. The $D_{W^\vee}$-structure on $\calR^W$ is then given by $\partial_{b_i}\mapsto\partial_{b_i}\otimes 1+1\otimes b_i^*$. Put
$$
f^V=\sum a_i\otimes a_i^*, \hskip.3in f^W=\sum b_i\otimes b_i^*
$$
which are the universal sections of the line bundles $\cO_{V^\vee}\boxtimes\calL$ and
$\cO_{W^\vee}\boxtimes(\calL\otimes\omega_X)$ respectively. By pulling them back to $V^\vee\times W^\vee\times X$, we shall view $f^V,f^W$ as sections of $\cO_{V^\vee\times W^\vee}\boxtimes\calL$ and
$\cO_{V^\vee\times W^\vee}\boxtimes(\calL\otimes\omega_X)$ respectively.

Recall that $U:=V^\vee\times X-\bL^\perp$ where $\bL^\vee$ is the total space of $\calL$, and let $\mathring{\bL}^\vee$ the complement of the zero section. 
As in the proof of Proposition \ref{general type classical solutions}, for given $b\in W^\vee$, $a\in V^\vee$, we can regard $f^W_b\Omega/f^V_a$
as a meromorphic form on $X$ with pole along $V(f_a)$.   As in \S\ref{CY case}, we have
$$\omega_{U\times W^\vee/V^\vee\times W^\vee}=(\cO_{V^\vee\times W^\vee}\boxtimes\omega_X)|_{U\times W^\vee}$$
as $D_{V^\vee\times W^\vee}\times\hat\frakg$-modules.

\begin{lem}\label{phi}
Define $\phi:(\calR^V\boxtimes\calR^W)\otimes k_{\beta_0}\to(\cO_{V^\vee\times W^\vee}\boxtimes\omega_X)|_{U\times W^\vee}$ by
$$
(a\otimes p)\boxtimes(b\otimes q)\mapsto{(-1)^l l! (f^W)^{1+(-1)^m\over2}\over (f^V)^{l+1} }(ab)\boxtimes(pq\Omega)
$$
where $l=\deg p$ and $m=\deg q\in\bZ/2\bZ$. Then $\phi$ is a $D_{V^\vee\times W^\vee}\times\hat\frakg$-module homomorphism, and it induces an isomorphisms of $D_{V^\vee\times W^\vee}$-modules
$$
\tau_{VW}\to \omega_{U\times W^\vee/V^\vee\times W^\vee}\otimes_{\hat\frakg}k.
$$
\end{lem}
\begin{proof}
It is a verbatim argument as in Lemma \ref{special:isom} and eqn. \eqref{id}.
\end{proof}

To complete our proof of Theorem \ref{general:main}, we observe that the proof  of Theorem \ref{special:main} carries over with just two changes: $V^\vee\times W^\vee$ and $\pi^\vee\times\id_{W^\vee}$ to replace $V^\vee$ and $\pi^\vee$ respectively.

As a consequence, for $\beta(e^V)=1$ and $\beta(e^W)=-1$ we have

\begin{cor}\label{rank formula}
Let $k=\C$, and $(a,b)\in V^\vee\times W^\vee$. Then the solution rank of $\tau_{VW}$ at $(a,b)$ is given by $\dim H_c^{0}(U_a, \calF|_{U_a})$, where $\calF=\on{Sol}(D_{X,\beta})$.
\end{cor}

\begin{proof}
This follows from a verbatim argument as in Corollary \ref{special:rank formula}.
\end{proof}

\section{A formula for $\tau$: general type case}\label{general beta general type}

We now return to the tautological system $\tau_{VW}=\tau_{VW}(G,X,\calL,\beta)$ introduced in Definition \ref{enhanced tau}.  We continue to use the notations introduced in \S\ref{general type hyperplane}. Let $\beta:\hat\frakg\equiv\frakg\oplus ke^V\oplus ke^W\ra k$ be a Lie algebra homomorphism. 
If $\beta(e^W)\neq0,-1$, then $\hat\tau_{VW}$ are zero. To see this, let $b_1^*,...,b_s^*$ denote a dual basis of $W^\vee$. 
Then in $\hat\tau_{VW}$, we have $b_i^*b_j^*\equiv 0$, hence
$$
0\equiv b_j^*(\sum_i -\partial_{b_i^*}b_i^*+\beta(e^W))=(1+\beta(e^W))b_j^*
$$
implying that $b_j^*\equiv 0$ for all $j$. But this implies that $\beta(e^W)\equiv0$, hence $\hat\tau_{VW}\equiv0$.
Now consider the case $\beta(e^W)=0$. Then $b_j^*\equiv0$ in $\hat\tau_{VW}$ as before. It follows that $\hat\tau_{VW}$ is supported on $V\times \{0\}$, and its inverse Fourier transform becomes
$$
\tau_{VW}=D_{V^\vee\times W^\vee}/D_{V^\vee\times W^\vee}J^V(\hat X)+D_{V^\vee\times W^\vee}(Z^V(\xi)+\beta(\xi),\xi\in\hat\frakg)+D_{V^\vee\times W^\vee}\widehat{W^\vee}
$$ 
where $J^V(\hat X)$ is the Fourier transform of the ideal of $X$ in $\P(V)$. This yields $\tau_{VW}=\tau\boxtimes\cO_{W^\vee}$, hence reducing $\tau_{VW}$ to a the special case of $\tau=\tau(G,X,\calL,\beta)$ introduced in \S\ref{Intro}.

From now on, we assume that 
$$\beta(e^W)=-1,\hskip.2in\beta(e^V)\notin\Z_{\leq 0}.
$$

In this section, we prove the following general type analogue of Theorem \ref{intro:general beta CY case}.

\begin{thm}\label{general beta general type case}
For $\beta(e^V)\notin\bZ_{\leq0}$ and $\beta(e^W)=-1$, there is a canonical isomorphism
$$\tau_{VW}\simeq p^!H^0\pi^\vee_+ev^!(D_{\mathring{\bL}^\vee,\beta})[1-r-s]$$
where $p:V^\vee\times W^\vee\ra V^\vee$ is the projection, $r=\dim V^\vee$ and $s=\dim W^\vee$.
\end{thm}

As in \S\ref{general type hyperplane}, the Fourier transform $\hat{\tau}_{VW}$ is a $D$-module on $V\times W$. 
Consider the open embedding 
\[\xymatrix{
j:V\times\mathring{W}\ar@{^(->}[r]&V\times W
}\]
and closed embedding 
\[i:V\rightarrow V\times W.\]
Then we have the following distinguished triangle 
\eqref{triangle}
\begin{equation}
i_{+}i^!\hat{\tau}_{VW}\rightarrow\hat{\tau}_{VW}\rightarrow j_{+}j^!\hat{\tau}_{VW}\rightarrow .
\end{equation}

Since $b_i^*b_j^*\equiv 0$ in $\hat{\tau}_{VW}$ for any $i,j$, it follows that on $\mathring{W}$, we have $b_i^*\equiv0$ for any $i$. Hence $\beta(e^W)\equiv\sum_{i=1}^s\partial_{b_i^*}b_i^*\equiv0$ in $\hat{\tau}_{VW}$.  But since $\beta(e^W)=-1$, we have $j^!\hat{\tau}_{VW}=0$, 
hence 
\begin{equation}\label{tau VW}
\hat{\tau}_{VW}\simeq i_+H^0i^!\hat{\tau}_{VW}.
\end{equation} 

Our main observation here is that we can compute the $D$-module $H^0i^!\hat{\tau}_{VW}$ in a way that is parallel to our computation in the CY case of $\hat\tau$ in \S\ref{exact seq}. To proceed, first we have the following analogue of Lemma \ref{restriction} for general types.

\begin{lem}\label{restriction:general type}
Let $\beta'(\xi)=\beta(\xi)-\tr_W Z^W(\xi)-\tr_V Z^V(\xi)$, $\xi\in\hat\frakg$. Then
\begin{equation}
(H^0i^!\hat{\tau}_{VW})|_{\mathring{V}}\simeq \mathring{\imath}_+D_{\mathring{X},\beta'}
\end{equation}
\end{lem}
\begin{proof}
$H^0i^!\hat{\tau_{VW}}$ consists of elements of $\hat{\tau_{VW}}$ annihilated by all $b_i^*$. One finds that they are precisely the elements that can be written in the form $\sum s_j\otimes b_j^*$, where $s_j\in D_V$. 

\blue{To check this, a straightforward yet not entirely trivial calculation is needed. The element $-\sum_{i=1}^s\partial_{b_i^*}b_i^*-1$ is killed by all the $b_i^*$, however, it is the Fourier transform of the Euler operator on $W$, which is in the kernel of $\hat{\tau}$ as we assumed $\beta(e^W)=-1$.}

On the other hand, we have
\begin{equation}\label{ix}
\mathring{\imath}_+D_{\mathring{X},\beta'}=(D_{\mathring{V}}/D_{\mathring{V}}I(X)\otimes_{\calO_{\mathring{V}}}\omega_{\mathring{X}/\mathring{V}})\otimes_{D_{\mathring{X}}}D_{\mathring{X},\beta'}
\end{equation}
where $\omega_{\mathring{X}/\mathring{V}}$ as a left $\calO_{\mathring{V}}$-module is generated by global sections $b_1^*,...,b_s^*$, (under the canonical identification $\omega\rightarrow\omega\wedge \frac{dt}{t}$) and relations among these generators as a left $\calO_{\mathring{V}}$-module are precisely given by $I^{VW}=\oplus_{d>0}\calI_d$ (see section \ref{general type hyperplane} for notations). Note that $\hat\frakg$ acts on $D_{\mathring{V}}\otimes_{\calO_{\mathring{V}}} \calO_{\mathring{V}}\left\langle b_1^*,...,b_s^*\right\rangle$ from the right via tensor product. The action descends to an action on 
\begin{equation}\label{ex}
D_{\mathring{V}}/D_{\mathring{V}}I(X)\otimes_{\calO_{\mathring{V}}}\calO_{\mathring{V}}\left\langle b_1^*,...,b_s^*\right\rangle/I^{VW}=D_{\mathring{V}}/D_{\mathring{V}}I(X)\otimes_{\calO_{\mathring{V}}}\omega_{\mathring{X}/\mathring{V}}. 
\end{equation}
One checks that this action coincides with the right $\hat{\frakg}$ action on $\omega_{\mathring{X}/\mathring{V}}$ through its right $D_{\mathring{X}}$-module structure given by negative Lie derivatives. Thus we have 
\begin{equation}\label{bridge}
(D_{\mathring{V}}/D_{\mathring{V}}I(X)\otimes_{\calO_{\mathring{V}}}\omega_{\mathring{X}/\mathring{V}})\otimes_{D_{\mathring{X}}}D_{\mathring{X},\beta'}\simeq (D_{\mathring{V}}/D_{\mathring{V}}I(X)\otimes_{\calO_{\mathring{V}}}\calO_{\mathring{V}}\left\langle b_1^*,...,b_s^*\right\rangle/I^{VW}\otimes k_{\beta'})\otimes_{\hat\frakg} k
\end{equation}
where $\hat\frakg$ acts on $(D_{\mathring{V}}/D_{\mathring{V}}I(X)\otimes_{\calO_{\mathring{V}}}\calO_{\mathring{V}}\left\langle b_1^*,...,b_s^*\right\rangle)/I^{VW}$ explicitly as explained.

Furthermore, from definition and the explanation in the beginning of the proof,
\begin{equation}\label{def}
(H^0i^!\hat{\tau}_{VW})|_{\mathring{V}}=(D_{\mathring{V}}/D_{\mathring{V}}I(X)\otimes_{\calO_{\mathring{V}}}\calO_{\mathring{V}}\left\langle b_1^*,...,b_s^*\right\rangle/I^{VW}\otimes k_{\beta'})\otimes_{\hat\frakg}k
\end{equation}

Combining \eqref{ix} \eqref{bridge}, and \eqref{def}, the lemma is proved.
\end{proof}

Next, by specializing \eqref{triangleII}, we have the following analogue of the sequence \eqref{4-term}:
\begin{equation}\label{general type:4-term}
0\rightarrow i_{0,+}H^{-1}i_0^+(H^0i^!\hat{\tau}_{VW})\rightarrow H^0j_{0,!}(H^0i^!\hat{\tau}_{VW})|_{\mathring{V}}
\rightarrow H^0i^!\hat{\tau}_{VW}\rightarrow i_{0,+}H^{0}i_0^+(H^0i^!\hat{\tau}_{VW})\rightarrow 0.
\end{equation}
With $\hat\tau$ in the CY case now replaced by $H^0i^!\hat{\tau}_{VW}$, Lemmas \ref{vanishing l} and \ref{for d} carry over readily to the general type case, with the following changes. The $!$-fiber of $\tau$ at $a\in V^\vee$ in the CY case is replaced by the $!$-fiber of $\tau_{VW}$ at $(a,b)\in V^\vee\times W^\vee$ in the general type case. The latter is now given in a parallel way by the Lie algebra homology of $\hat\frakg$ with coefficients in $S^{VW}\otimes\beta$, where
$$
S^{VW}:=S^V\otimes S^W/S^V\otimes S^W\calI_1.
$$
Here the $\hat\frakg$-action on $S^V=\calO_V/I(\hat X)$ is given verbatim as in \S\ref{shrik r} as the sum of two actions $\alpha_1,\alpha_2$ (see before Lemma \ref{vanishing r}.) The $\hat\frakg$-action on $S^W=\cO_W/\cO_W\Sym^2 W^\vee$ is given by $Z^{W^\vee}:\hat\frakg\ra\End W^\vee\ra\End S^W$, where $e^W$ acts trivially ($\beta(e^W)=-1$). This shows that the first and last terms of \eqref{general type:4-term} are both zero, hence
\begin{equation*}
H^0j_{0,!}(H^0i^!\hat{\tau}_{VW})|_{\mathring{V}}\simeq H^0i^!\hat{\tau}_{VW}.
\end{equation*}
Together with Lemma \ref{restriction:general type}, this implies that
\begin{equation}\label{1}
\four(H^0i^!\hat{\tau}_{VW})\simeq \four( H^0j_{0,!}(H^0i^!\hat{\tau}_{VW})|_{\mathring{V}})\simeq H^0
\four(j_{0,!}\mathring{\imath}_+D_{\mathring{X},\beta'}).
\end{equation}

Next, to compute the right hand side, observe that \eqref{aux1}-\eqref{aux3} and Lemma \ref{aux4} hold for an arbitrary very ample line bundle $\calL$.
This yields
\begin{equation}\label{2}
\four(j_{0,!}\mathring{\imath}_+D_{\mathring{X},\beta'})
\simeq\pi_+^\vee ev^!D_{\mathring{\bL}^\vee,\beta}\left[1-r\right].
\end{equation}
Finally, since $p$ is dual to the inclusion $i$, and combining \eqref{tau VW} and \eqref{1}-\eqref{2}, it follows that
\begin{equation*}
\tau_{VW}\simeq\four(i_+H^0i^!\hat{\tau_{VW}})\simeq p^!\four(H^0i^!\hat{\tau_{VW}})\left[-s\right]\simeq p^!H^0\pi_+^\vee ev^!D_{\mathring{\bL}^\vee,\beta}\left[1-r-s\right]
\end{equation*}
This completes the proof of Theorem \ref{general beta general type case}.

\section{Projective homogeneous spaces}\label{G/P}
To apply our results, we need to understand the D-module $ev^!(D_{\mathring{\bL}^\vee,\beta})[1-r]$ in various situations.
In this section, we assume that $G$ is semisimple and $X$ is a projective homogeneous $G$-variety, i.e. $X$ is a partial flag variety, and $\beta(e)=1$. Theorem \ref{special:main} takes a particularly easy form in this case. We first have

\begin{cor}\label{tau is GM connection}
If $\beta({\frakg})=0$ and $X$ is a homogeneous $G$-variety, then $\tau\simeq H^0\pi^\vee_+\calO_U$.
\end{cor}
\begin{proof} Recall Corollary \ref{surjective}. Then if $X$ is homogeneous, $\frakg\otimes\calO_X\to T_X$ is surjective. Therefore $D_{X,0}=D_X/D_XT_X=\calO_X$.
\end{proof}

Note that this corollary implies that a tautological system in this case, which is a priori defined as a D-module by generators and relations, is of geometric origin, i.e. itself is a Gauss-Mannin connection.

\begin{cor}\label{main conjecture}
Conjecture \ref{holo-rank} holds.
\end{cor}

For general types, Theorem \ref{general:main} also specializes in an analogous way, and we get the following analogues of both Corollaries \ref{tau is GM connection} and \ref{main conjecture}.
\begin{cor}
If $\beta({\frakg})=0$ and $X$ is a homogeneous $G$-variety, then\\ $\tau\simeq H^0(\pi^\vee\times \id_{W^\vee})_+\calO_{U\times W^\vee}$.
\end{cor}

\begin{cor}
Let $X$ be an $n$-dimensional projective homogeneous space of a semisimple group $G$. Assume $\beta(e^V)=1$ and $\beta(e^W)=-1$. Then the solution rank of $\tau_{VW}=\tau_{VW}(G,X,\calL,\beta)$ at $(a,b)\in V^\vee\times W^\vee$  is given by $\dim H_n(X-Y_a)$.
\end{cor}

We can also describe a rank 1 point for a general homogeneous variety $X$ in the case of $\calL=\omega_X^{-1}$, using the projected Richardson stratification of $X$ studied in \cite{Lusztig98}\cite{Reitsch05}\cite{Hague10}\cite{KLS10}.

We follow the notations in \cite{KLS10}. Let $G$ be a reductive algebraic group over an algebraically closed field $k$ of characteristics zero, $B$ a Borel subgroup and  $P\supset B$ a parabolic subgroup in $G$. Put $B^+=B$ and let $B^-$ be the opposite Borel subgroup. Let $Q(W,W_P)$ be the set of equivalence classes of $P$-Bruhat intervals \cite[\S2]{KLS10}. Each equivalence class is uniquely specified by a pair $(u,w)$ of elements in the Weyl group. For $(u,w)\in Q(W,W_P)$, put $X^w_u:=(B^-uB/B)\cap(B^+wB/B)$, an open Richardson variety in $G/B$.

\begin{prop} \cite[\S7]{Reitsch05}\cite[\S3]{KLS10}
There is a stratification of $X=G/P$ of the form $X=\coprod_{(u,w)\in Q(W,W_P)}\mathring{\Pi}^w_u$, where each stratum $\mathring{\Pi}^w_u$ is the isomorphic image of $X_u^w$ under the natural projection $G/B\ra G/P$.
\end{prop}

The next result and proof are communicated to us by T. Lam.
\begin{prop}
Let $\Pi_1,..,\Pi_s$ be the closures of the codimension 1 strata in $X$. Then $\cup_i\Pi_i$ is an anticanonical divisor in $X$, and its complement in $X$ has one dimensional middle cohomology.
\end{prop}
\begin{proof}
The first assertion follows from Lemma 5.4 \cite{KLS10}. Since $X- \cup_i\Pi_i$ is the largest stratum, it is isomorphic to
an open Richardson variety $X_u^w$ in $G/B$. 
It is well-known that (see for example \cite{ravi})
\begin{equation}
H^N_c(X_u^w) = \Hom(M_u,M_w)
\end{equation}
where $N=\dim X_u^w$ and $M_w$ denotes the Verma module of the Lie algebra of $G$ of highest weight $-w(\rho)-\rho$.
By combining Theorems 1-4 \cite{BGG}, or by the Kazhdan-Lusztig conjecture, one has
\begin{equation}\label{mult}
\dim \Hom(M_u,M_w) = 1
\end{equation}
\end{proof}

Therefore, by Theorem \ref{holo-rank-thm} we have
\begin{cor}
Let $a\in\Gamma(X,\omega_X^{-1})$ be the defining section of the anticanonical divisor $\cup_i\Pi_i$. Then $a$ is a rank 1 point of $X$.
\end{cor}

\begin{rmk}
This section is torus invariant. Due to a theorem of Kostant that later generalized by Luna, If a point $f_a$ in $V^{\vee}$ is fixed by a reductive subgroup $H$, then $Gv$ is closed if and only if $C_G(H)v$ is closed. If $H$ is the maximal torus, then $C_G(H)=H$, so the orbit is closed, therefore this section is GIT semistable w.r.t. the action of $G$ on $V^\vee$.
\end{rmk}

\begin{ex}
Consider the Grassmannian $X=G(d,n)$. According to \cite{KLS11}, $\cup_i\Pi_i$ is defined by the section $a=x_{1,2,..,d}x_{2,3,...,d+1}...x_{n,1,..,d-1}$, where the $x_{i_1,..,i_d}$ are the Pl\"ucker coordinates of $X$. This generalizes a construction in \cite{BHLSY} for $d=2$.
\end{ex}

\section{Rank 1 points of 1-step flags}\label{partial1}
\red{ Maybe emerge the following two sections in Sect. \ref{G/P}} \green{ I think logically those 3 sections do share the same background. But one concern I have is that this may make \ref{G/P} far too long - almost half the paper - and so reader may find it difficult to sort out what's in it. One possibility is to make 3 subsections within one section. But not sure how you and An feel about this.}
{\it Notation.} If $m$ is an $p\times q$ matrix, and $J\subset(1,2,..,p)$ is an ordered index set, then $m_J$ denotes the submatrix of $m$ given by the rows labelled by $J$, and we also call $m_J$ the $J$-block of $m$.
We denote by $x_J$, $J\subset(1,2,..,n)$, the Pl\"ucker coordinates of the $d$-plane Grassmannian $F(d,n)$. Let $M$ be the space of rank $d$ matrices of size $n\times d$. Then $GL_d$ acts freely and properly on $M$ by right multiplication and $M/GL_d\simeq F(d,n)$. Under this identification, we denote the projection map of the Stiefel bundle $M\ra X$ by $m\mapsto[m]:=m\cdot GL_d$. Then $x_J$ can be viewed as the function $x_J:M\ra\C$, $m\mapsto\det(m_J)$. Given a section $f$ of any line bundle on $X$, we denote by $X(f)$ the complement of $f=0$ in $X$, and by $M(f)$ the preimage of $X(f)$ under $M\ra X$.

\begin{prop}\label{1step}
The 1-step flag variety $X=F(d,n)$ admits a rank 1 point $f\in\Gamma(X,\omega_X^{-1})$ such that $(x_J)^k\vert f$ for some $J\subset(1,2,..,n)$ with $|J|=d$ and $k=\min(d,n-d)$. If $n=2d$, then $f=(x_{1,..,d})^d (x_{d+1,..,n})^d$ is a rank 1 point.
\end{prop}

\begin{proof}

(a) Consider the case $n\geq l+d\geq 2d$. We have
\begin{align}\label{embedding-1step}
&X_1:=F(d,n-l)\into X, ~~~E\mapsto E\oplus 0_l\cr
&X_2:=F(d,l)\into X,~~~E\mapsto 0_{n-l}\oplus E.
\end{align}
Here we view $\C^n=\C^{n-l}\oplus\C^l$.
Let $f_1$ be a given rank 1 point of $X_1$ such that $(x_{1,..,d})^{k_1}| f_1$, $k_1=\min(d,n-l-d)$, and $f_2$ a rank 1 point of $X_2$ such that $(x_{n-d+1,..n})^{k_2}\vert f_2$, $k_2=\min(d,l-d)$. (In case $l=d$,  $X_2=\text{pt}$, we simply take $f_2=(x_{n-d+1,..,n})^d$; in case $n-l=d$, $X_1=\text{pt}$, we take $f_1=(x_{1,..,d})^d$). We can view $f_1,f_2$ as sections of $\cO_X(n-l)$ and $\cO_X(l)$ respectively on $X=F(d,n)$. Then the restriction of $f_1$ to $X_1$ under \eqref{embedding-1step} becomes a section of $\cO_{X_1}(n-l)$. Likewise the restriction of $f_2$ to $X_2$ becomes a section of $\cO_{X_2}(l)$. We claim that $f=f_1f_2\in\Gamma(X,\omega_X^{-1})$ is a rank 1 point of $X$. We will first construct an explicit isomorphism
$$X_1(f_1)\times X_2(f_2)\times GL_d\ra X(f).$$
Let $M_1,M_2,M$ be the Stiefel bundles over $X_1,X_2,X$ respectively. Since $x_{J}\vert f_2$, $J=(n-d+1,..n)$, each $m_2'\in M_2(f_2)$ has a nonsingular $J$-block $D$. Define
$$
M_1(f_1)\times M_2(f_2)\ra M(f),~~~
m_1',m_2'\mapsto m=\left[\begin{matrix}m_1'D\cr m_2'\end{matrix}\right].
$$
This is well-defined since
$$f(m)=f_1(m_1'D)f_2(m_2')=(\det D)^{n-l}f_1(m_1')f_2(m_2').$$
The map is a bijection with inverse $m=\left[\begin{matrix}m_1\cr m_2\end{matrix}\right]\mapsto m_1(m_2)_J^{-1},m_2$. Now let $h\in GL_d$ act on $M_1(f_1)\times M_2(f_2)$ by the formula $(m_1',m_2'h^{-1})$. Then the map is equivariant. It follows that we have an isomorphism
$$
M_1(f_1)\times X_2(f_2)\ra X(f).
$$
Finally,  since $x_{1,..,d}\vert f_1$ each $m_1'\in M_1(f_1)$ has a nonsingular top $d\times d$ block. It follows that the principal bundle $GL_d-M_1(f_1)\ra X_1(f_1)$ is trivial. In fact, it has a (unique) section of the form $[m_1]\mapsto m_1'$ where $m_1'$ is the unique representative in $[m_1]$ whose top $d\times d$ block is the identity matrix $I_d$. This proves that
$$
X(f)\simeq X_1(f_1)\times X_2(f_2)\times GL_d.
$$
Since the $X_i(f_i)$ are affine varieties, all de Rham cohomology of degree $>\dim X_i$, vanishes. Since $f_1,f_2$ are rank 1 points of $X_1,X_2$ respectively, we have $H^{\dim X_i}(X_i(f_i))=\C$ by Theorem \ref{holo-rank-thm}. It follows that
$$
H^{\dim X}(X(f))\simeq H^{\dim X_1}(X_1(f_1))\otimes H^{\dim X_2}( X_2(f_2))\otimes H^{d^2}(GL_{d})\simeq\C.
$$
So $f$ is a rank 1 point of $X$ such that $(x_{1,..,d})^{k_1}(x_{n-d+1,..n})^{k_2}|f$.

(b) To complete the proof of the proposition, we proceed by induction. For $X=F(1,2)=\P^1$, paragraph (a) with $n=2$ and $l=d=1$ shows that $x_1x_2$ is a rank 1 point of $X$, and the proposition holds. Assume that it holds for up to $F(d,n-1)$, and consider the case $X=F(d,n)$.  For $n<2d$ we have $F(d,n)\simeq F(n-d,n)$, in which case paragraph (a) with $l,d$ playing the role of $d,n-d$, yields a rank 1 point $f$ of $F(n-d,n)$ with $(x_J)^{n-d}|f$ and $|J|=n-d$. This in turn yields a rank 1 point of $F(d,n)$ divisible by $(x_{J^c})^{n-d}$ where $J^c=(1,..,n)-J$. For $n=2d$, paragraph (a) with $n=l+d=2d$ shows that $(x_{1,..,d})^d (x_{d+1,..,n})^d$ is a rank 1 point of $X$. For $n>2d$, paragraph (a) with $l=d$ and our inductive hypothesis shows that $X$ has a rank 1 point $f=f_1\cdot(x_{n-d+1,..,n})^d$, where $f_1$ is a rank 1 point of $F(d,n-d)$. This completes the proof.
\end{proof}

\begin{cor}
Let $n=l_1+\cdots+l_s$ be a partition of $n$ with $l_p\geq d$. Let $f_p$ be a rank 1 point of $F(d,l_p)\into F(d,n)$, viewed as a degree $l_p$ polynomial in the Pl\"ucker coordinates $x_J$ of $X=F(d,n)$ with $J\subset(l_1+\cdots+l_{p-1}+1,..,l_1+\cdots+l_{p})$ and $\vert J\vert=d$, such that $(x_{l_1+\cdots+l_{p-1}+1,..,l_1+\cdots+l_{p-1}+d})\vert f_p$. Then $f=f_1\cdots f_s$ is a rank 1 point of $X$. In fact, we have an isomorphism
$$
X(f)\simeq X_1(f_1)\times\cdots\times X_s(f_s)\times (GL_d)^{s-1}
$$
where $X_p:=F(d,l_p)$.
\end{cor}
\begin{proof}
Start with $l=l_2+\cdots+l_s$. Then paragraph (a) in the preceding proof gives
$$
X(f)\simeq X_1(f_1)\times X_2'(f_2\cdots f_s)\times GL_d
$$
where $X_2':=F(d,n-l_1)$. Now the result follows by induction on $s$.
\end{proof}

\section{Rank 1 points of $r$-step flags}\label{partial2}

{\it Throughout this section, let $X=F(d_1,..,d_r,n)$ be the $r$-step flag variety with $r\geq2$.} We will give a recursive procedure that produces a rank 1 point of $X$, by assembling rank 1 points of lower step flag varieties. We begin with some notations and terminology.

Let $\cO_i(1)$ be the standard hyperplane bundle on $F(d_i,n)$. The space of its sections is an irreducible $G=SL_n$ module of highest weight $\lambda_{d_i}$, the $d_i$th fundamental weight of $G$. We shall denote by $\lambda_{d_i}$ the pullback of $\cO_i(1)$ via the composition map
$$
X\into F(d_1,n)\times\cdots\times F(d_r,n)\onto F(d_i,n)
$$
where the first is the incidence embedding and the second is the $i$th projection. Then $\Pic(X)$ is the free abelian group generated by $\lambda_{d_1},...,\lambda_{d_r}$. We also have (see \cite{LSY})
\begin{equation}\label{KX}
-K_X=\omega_X^{-1}=(n-d_{r-1})\lambda_{d_r}+(d_r-d_{r-2})\lambda_{d_{r-1}}+\cdots+(d_3-d_1)\lambda_{d_2}+d_2\lambda_{d_1}.
\end{equation}
By the Borel-Weil theorem, the restriction map
$$
\Gamma(F(d_1,n),\cO_1(k_1))\otimes\cdots\otimes\Gamma(F(d_r,n),\cO_r(k_r))\ra\Gamma(X,\sum_i k_i\lambda_{d_i})
$$
is a $G$-equivariant surjective map for any $k_1,..,k_r\in\Z$ (and both spaces are zero unless $k_i\geq0$ for all $i$). Thus any homogeneous polynomial in  the Pl\"ucker coordinates $x_{J_i}$ with $\vert J_i\vert=d_i$, of multi-degree $(k_1,..,k_r)\in\Z^r_{\geq0}$, can be viewed as a section of the line bundle $\sum_i k_i\lambda_{d_i}$ on $X$. Conversely, any section of this line bundle on $X$ can be expressed as such a polynomial (not necessarily unique).

Let $k<n-d_r$ and consider the embeddings
\begin{align}\label{embeddings}
&X_1:=F(d_1,..,d_r,n-k)\into X,~~~E^\bullet\mapsto E^\bullet\oplus0_k\cr
&X_2:=F(d_1-k,..,d_r-k,n-k)\into X,~~~E^\bullet\mapsto E^\bullet\oplus\C^k.
\end{align}
Here we view $\C^n=\C^{n-k}\oplus\C^k$, and $X_1,X_2$ are viewed as spaces consisting of $r$-step flags in the factor $\C^{n-k}$. For each Pl\"ucker coordinate $x_{J'}$ on $X_1$ with $J'\subset(1,2,...,n-k)$, is the restriction of $x_{J'}$, regarded as a Pl\"ucker coordinate on $X$. Likewise, any homogeneous polynomial $f_1$ in the $x_{J'}$, can be viewed as the restriction of a section $\bar f_1$ on $X$ involving only the same Pl\"ucker coordinates. We shall often impose certain divisibility conditions (called the hyperplane property -- see below) on $\bar f_1$, but will state them in terms of $f_1$. Similarly each Pl\"ucker coordinate $x_{J'}$ on $X_2$ is the restriction of $x_{J'\cup(n-k+1,..,n)}$ on $X$; any given homogeneous polynomial $f_2$ in the $x_{J'}$, is the restriction of a section $\tilde f_2$ on $X$ involving only the $x_{J'\cup(n-k+1,..,n)}$. Again, divisibility conditions imposed on $\tilde f_2$ will be stated in terms of $f_2$.

As in the case of 1-step flags, we can view $X=M/H$, where
$$H:=GL_{d_r}\times\cdots\times GL_{d_1}$$
and $M$ is the space of $r$-tuple of matrices $m=(m_r,..,m_1)$, $m_i$ a $d_{i+1}\times d_i$ matrix of rank $d_i$ ($d_{r+1}\equiv n$), where $h=(h_r,..,h_1)\in H$ acts on $M$ by the formula
\begin{equation}\label{mh}
m\cdot h^{-1}:=(m_rh_r^{-1},h_rm_{r-1}h_{r-1}^{-1},..,h_2m_1h_1^{-1}).
\end{equation}
Under the identification $X=M/H$, we denote the projection map $M\ra X$ by  $m\mapsto[m]:=m\cdot H$, and call $M$ {\it the Stiefel bundle} over $X$. We can view a Pl\"ucker coordinate $x_J$, $|J|=d_i$, on $X$ as the function $x_J:M\ra\C$, $x_J(m)=\det(m_r\cdots m_i)_J$. In particular, $f_1$ is a section on $X_1$ and $\bar f_1$ a section on $X$ restricting to it as described above, then for $J=(1,..,n-k)$ we have
$$\bar f_1(m_r,..,m_1)=f_1((m_r)_J,m_{r-1},..,m_1)$$
whenever $\rk (m_r)_J=d_r$.  Let $m=(m_r,..,m_1)\in M$ where the $m_i$ have the form
$$
m_i=\left[\begin{matrix}m_i'& *\cr O& I_k\end{matrix}\right]
$$
where $I_k$ is the $k\times k$ identity matrix and $O$ a zero block. Then $x_{J'\cup(n-k+1,..,n)}(m)=\det(m_r'\cdots m_i')_{J'}$ for any $J'\subset(1,..,n-k)$ with $|J'|=d_i-k$. So, if $f_2$ is a section on $X_2$ and $\tilde f_2$ a section on $X$ restricting to it as described above, then
$$\tilde f_2(m_r,..,m_1)=f_2(m_r',...,m_1').$$

Let $f$ be a nonzero section of a line bundle on $X$, and let $X(f)$ be the complement of $f=0$, and $M(f)$ the preimage of $X(f)$ under $M\ra X$.

\begin{defn} (Hyperplane property)
We say that $f$ has the hyperplane property if for some $J_i\subset(1,2,..,n)$ with $|J_i|=d_i$, $i=1,..,r$, we have $(x_{J_1}\cdots x_{J_r})|f$. In other words, the hypersurface $f=0$ contains the union of hyperplanes $x_{J_i}=0$.
\end{defn}
Note that if $f$ has the hyperplane property, we can always find a suitable permutation matrix $g\in GL_n$ such that the $g$-translate of $f$ has the hyperplane property where $J_r=(n-d_r+1,..,n)$. In the construction that follows, we will often arrange our section $f$ so that this occurs. Next, we have the following elementary lemma.

\begin{lem}\label{special-section}
Assume $f$ has the hyperplane property $(x_{J_1}\cdots x_{J_r})| f$. Then the principal $H$-bundle $M(f)\ra X(f)$, has a unique section $m=(m_r,..,m_1)$, where the $m_r,..,m_1$ are matrix valued functions on $X(f)$ such that
$$
(m_r\cdots m_i)_{J_i}=I_{d_i}.
$$
\end{lem}

\begin{defn} (Special section)
We call the section given in Lemma \ref{special-section}, the special section of $M(f)$ (which depends on the index sets $J_1,..,J_r$).
\end{defn}

We now describe our recursive procedure that produces a rank 1 point of $X$ with the hyperplane property.

\noi{\bf Case 1.} Assume $d_{r-1}+d_r<n$.
Consider (cf. \eqref{embeddings})
\begin{align*}
&X_1:=F(d_1,..,d_{r-1},d_r)\into F(d_1,..,d_{r-1},n),~~~E^\bullet\mapsto 0_{n-d_r}\oplus E^\bullet\cr
&X_2:=F(d_r,n).
\end{align*}
Let $M_1,M_2,M$ be the Stiefel bundles over $X_1,X_2,X$ respectively. Let $f_1,f_2$ be rank 1 points of $X_1,X_2$ respectively with the hyperplane properties
\begin{equation}\label{f1f2-rstepcase1}
(x_{J_1}\cdots x_{J_{r-1}})\vert f_1,~~~(x_{J_r})^k\vert f_2
\end{equation}
for some $J_i$ with $\vert J_i\vert=d_i$, $i=1,...,r$, and $J_1=(1,..,d_1)$, $J_r=(n-d_r+1,..,n)$, $k=\min(d_r,n-d_r)>d_{r-1}$. Such an $f_2$ exists by Proposition \ref{1step}. Put
\begin{equation}\label{f-rstepcase1}
f=\bar f_1\cdot\bar f_2\cdot(x_{J_r})^{-d_{r-1}}.
\end{equation}
Then we have
\begin{equation}\label{hyperplane-rstepcase1}
(x_{J_1}\cdots x_{J_r})\vert f
\end{equation}
It follows easily from \eqref{KX} that $f$ is a section of $\omega_X^{-1}$.

\begin{lem}
We have an $H=GL_{d_r}\times\cdots\times GL_{d_1}$ equivariant isomorphism
\begin{align*}
&M_1(f_1)\times M_2(f_2)\ra M(f)\cr
&(m_{r-1}',..,m_1'),m_r'\mapsto m=(m_r',D^{-1}m_{r-1}',m_{r-2}',..,m_1')
\end{align*}
where $D$ is the $J_r$-block of $m_r'$. Therefore the map descends to an isomorphism $X_1(f_1)\times X_2(f_2)\ra X(f)$.
\end{lem}

\begin{proof}
For $m_r'\in M_2(f_2)$, its $J_r$-block $D$ is a nonsingular matrix in $GL_{d_2}$ since  $(x_{J_r})^k\vert f_2$.  Suppose $f_1(m_{r-1}',..,m_1')f_2(m_r')\neq 0$. Then
$$
f(m)=f_1((m_r'D^{-1}m_{r-1}')_{J_r},m_{r-2}',..,m_1')f_2(m_r')(\det(m_r')_{J_r})^{-d_{r-1}}.
$$
Since $(m_r')_{J_r}=D$, it follows that $(m_r'D^{-1}m_{r-1}')_{J_r}=m_{r-1}'$ and we have
$$f(m)=f_1(m_{r-1}',..,m_1')f_2(m_r')(\det D)^{-d_{r-1}}\neq 0.$$
So, the map is well-defined. Now, $h=(h_r,..,h_1)\in H$ acts on $M(f)$ by \eqref{mh}, and on $M_1(f_1)\times M_2(f_2)$ by the formula
$$m_rh_r^{-1},~(m_{r-1}h_{r-1}^{-1},h_{r-1}m_{r-2}h_{r-2}^{-1},.., h_2m_1h_1^{-1}).$$
Therefore our map is $H$-equivariant. Moreover, the map
$$
M(f)\ra M_1(f_1)\times M_2(f_2),~~
(m_r,..,m_1)\mapsto ((m_r)_{J_r}m_{r-1},m_{r-2},..,m_1),m_r
$$
is well-defined and is the inverse of the map above.
\end{proof}

The lemma and Theorem \ref{holo-rank-thm} imply

\begin{prop}\label{rstepcase1}
For $d_{r-1}+d_r<n$,  if any $s$-step flag variety for $s<r$ admits a rank 1 point with the hyperplane property, then $X=F(d_1,..,d_r,n)$ admits one as well.
\end{prop}

By Proposition \ref{1step}, for $d_1+d_2<n$ it follows that $F(d_1,d_2,n)$ admits a rank 1 point with the hyperplane property.
This also implies that for $d_1+d_2>n$, then $F(d_1,d_2,n)\simeq F(n-d_2,n-d_1,n)$ admits one as well.

\noi{\bf Case 2.} Assume $d_{r-1}+d_r=n$ and $r=2$. Consider the following section of $\omega_X^{-1}$:
$$
f=(x_{1,..,d_1})^{d_2}(x_{d_1+1,..,n})^{d_2}.
$$
Then by Lemma \ref{special-section}, the special section of $M(f)\ra X(f)$ has the form
$$
m=(m_2,m_1)=(\left[\begin{matrix}A_2 \cr I_{d_2}\end{matrix}\right], m_1)~~~\text{such that}~A_2m_1=I_{d_1}.
$$
Since $m_1(o)$ has rank $d_1$ at each point $o\in X(f)$, the second equation shows that the function
$$
m_1:X(f)\ra M_1, ~~~o\mapsto m_1(o)
$$
is onto. Here $M_1$ be the Stiefel bundle over $F(d_1,d_2)$. Moreover, the level set of this function at each point is
an affine space of dimension $d_1d_2-d_1^2$. It follows that $X(f)$ is homotopy equivalent to $M_1$. Finally, the principal $GL_{d_1}$-bundle $M_1\ra F(d_1,d_2)$ is over a simply connected base. Thus by the Serre spectral sequence, the highest degree nonzero cohomology group of $M_1$ is one dimensional at degree $2d_1d_2-d_1^2=\dim X$. By Theorem \ref{holo-rank-thm}, we have

\begin{prop}\label{2stepcase2}
For $d_1+d_2=n$, $X=F(d_1,d_2,n)$ admits the rank 1 point $f=(x_{1,..,d_1})^{d_2}(x_{d_1+1,..,n})^{d_2}$.
\end{prop}

\begin{rem}
The propositions in Cases 1-2 ($r=2$) now imply that any 2-step flag variety $F(d_1,d_2,n)$ admits a rank 1 point with the hyperplane property.
\end{rem}

\noi{\bf Case 3.} Assume $d_{r-1}+d_r=n$ and $r\geq3$.
Consider (cf. \eqref{embeddings})
\begin{align*}
&X_1:=F(d_1,..,d_{r-2},d_{r-1})\into F(d_1,..,d_{r-2},n),~~~E^\bullet\mapsto E^\bullet\oplus 0_{n-d_{r-1}}\cr
&X_2:=F(d_{r-1},d_r,n).
\end{align*}
Let $M_1,M_2,M$ be the Stiefel bundles over $X_1,X_2,X$ respectively. Let $f_1,f_2$ be rank 1 points of $X_1,X_2$ respectively with the hyperplane properties
\begin{equation}\label{f1f2-rstepcase3}
(x_{J_1}\cdots x_{J_{r-2}})\vert f_1,~~~~~f_2=(x_{J_{r-1}})^{d_r}(x_{J_r})^{d_r}
\end{equation}
for some $J_i$ with $\vert J_i\vert=d_i$, $i=1,...,r$, and $J_1=(1,..,d_{r-1})$, $J_r=(n-d_r+1,..,n)$. Note that $f_2$ is given by Proposition \ref{2stepcase2}. Put
\begin{equation}\label{f-rstepcase1}
f=\bar f_1\cdot\bar f_2\cdot(x_{J_{r-1}})^{-d_{r-2}}\in\Gamma(X,\omega^{-1}).
\end{equation}
Then we have
\begin{equation}\label{hyperplane-rstepcase3}
(x_{J_1}\cdots x_{J_{r-1}} (x_{J_r})^{d_r})\vert f
\end{equation}

Since $x_{J_{r-1}}\vert f_2$,  the $J_{r-1}=(1,..,d_{r-1})$-block $D$ of $m_r' m_{r-1}'$ for $(m_r',m_{r-1}')\in M_2(f_2)$ is nonsingular.

\begin{lem}
We have an $H=GL_{d_r}\times\cdots\times GL_{d_1}$ equivariant isomorphism
\begin{align*}
&M_1(f_1)\times M_2(f_2)\ra M(f)\cr
&(m_{r-2}',..,m_1'),(m_r',m_{r-1}')\mapsto m=(m_r',m_{r-1}',D^{-1}m_{r-2}',m_{r-3}',..,m_1')
\end{align*}
where $D$ is the $J_{r-1}=(1,..,d_{r-1})$-block of $m_r' m_{r-1}'$. Therefore the map descends to an isomorphism $X_1(f_1)\times X_2(f_2)\ra X(f)$.
\end{lem}
The proof is closely analogous to the lemma in Case 1, and will be omitted. The lemma and Theorem \ref{holo-rank-thm} imply

\begin{prop}\label{rstepcase3}
For $d_{r-1}+d_r=n$, if any $s$-step flag variety for $s<r$ admits a rank 1 point with the hyperplane property, then $X=F(d_1,..,d_r,n)$ admits one such $f$ that satisfies $(x_{J_r})^{d_r}\vert f$ where $J_r=(n-d_r+1,..,n)$.
\end{prop}

\noi{\bf Case 4.} Assume $d_1+d_2=n$. Then $X\simeq F(n-d_r,..,n-d_2,n-d_1,n)$, which belongs in Case 3, and the analogue of Proposition \ref{rstepcase3} is

\begin{prop}\label{rstepcase4}
For $d_1+d_2=n$,  if any $s$-step flag variety for $s<r$ admits a rank 1 point with the hyperplane property, then $X=F(d_1,..,d_r,n)$ admits one such $f$ that satisfies $(x_{J_1})^{d_1}\vert f$ where $J_1=(1,..,d_1)$.
\end{prop}

\noi{\bf Case 5.} Assume $d_{r-1}+d_r>n$. There exists a unique $a$ with $r>a>1$ such that $d_a+d_{a+1}>n\geq d_{a-1}+d_a$. Assume $n>2d_a$ first. We will consider $n=2d_a$ and $n=d_{a-1}+d_a$ in Cases 6-7 below separately. Consider
\begin{align*}
&X_1:=F(d_1,..,d_a,n-d_a)\into F(d_1,..,d_a,n),~~~(E_1^i)\mapsto(E_1^i\oplus 0_{d_a})\cr
&X_2:=F(d_{a+1}-d_a,..,d_r-d_a,n-d_a)\into F(d_{a+1},,..,d_r,n),~~~(E_2^j)\mapsto(E_2^j\oplus\C^{d_a}).
\end{align*}
Here we view $\C^n=\C^{n-d_a}\oplus\C^{d_a}$. Let $M_1,M_2,M$ be the Stiefel bundles over $X_1,X_2,X$ respectively. Let $f_1,f_2$ be rank 1 points of $X_1,X_2$ respectively with the hyperplane properties
\begin{equation}\label{f1f2-rstepcase5}
(x_{J_1}\cdots x_{J_a})\vert f_1,~~~(x_{J_{a+1}'}\cdots x_{J_r'})\vert f_2
\end{equation}
for some $J_i\subset(1,2,..,n-d_a)$ with $\vert J_i\vert=d_i$ ($i=1,...,a$) and $J_a=(n-2d_a+1,..,n-d_a)$, and for some $J_i'\subset(1,2,..,n-d_a)$ with $\vert J_i'\vert=d_i-d_a$ ($i=a+1,..,r$) and $J_r'=(n-d_r+1,..,n-d_a)$. Put $J:=(n-d_a+1,..,n)$, $J_i:=J_i'\cup J$, $i=a+1,..,n$, and
\begin{equation}\label{f-rstepcase5}
f:=\bar f_1\cdot\tilde f_2\cdot(x_J)^{d_{a+1}+d_a-n}.
\end{equation}
Then $f$ has the hyperplane property
\begin{equation}\label{hyperplane-rstepcase5}
(x_{J_1}\cdots\widehat{x_{J_a}}\cdots x_{J_r}x_J)\vert f.
\end{equation}

\begin{lem}
The special section $m=(m_r,..,m_1)$ (cf. Lemma \ref{special-section}) of $M(f)\ra X(f)$ has the following form:
\begin{align}\label{relations-rstepcase5}
m_i&=\left[\begin{matrix}m_i'& A_i\cr O& I_{d_a}\end{matrix}\right], ~~~i=a+1,..,r\cr
m_a&=\left[\begin{matrix}A_a\cr I_{d_a}\end{matrix}\right]\cr
m_r\cdots m_a&=\left[\begin{matrix}m_a'D\cr I_{d_a}\end{matrix}\right]\cr
m_{a-1}&=D^{-1}m_{a-1}'\cr
m_i&=m_i',~~~i=1,..,a-2
\end{align}
where $D$ is a $GL_{d_a}$-valued function, $A_a,..,A_r$ are matrix valued functions, and $(m_a',..,m_1')$, $(m_r',..,m_{a+1}')$ are matrix valued functions taking values in the special sections of the $M_1(f_1)\ra X_1(f_1)$, $M_2(f_2)\ra X_2(f_2)$ respectively.
\end{lem}

\begin{proof}
For $o\in X(f)$, we will write $m_i\equiv m_i(o)$, $m_i'\equiv m_i'(o)$, $D\equiv D(o)$, etc. Then $m=m(o)\in M(f)$ means that
$$
0\neq f(m)=\bar f_1(m_r\cdots m_a,m_{a-1},..,m_1)\tilde f_2(m_r,..,m_{a+1}) \det(m_r)_{J_r}.
$$

{(a)} Since $x_{J_r'}\vert f_2$, we have $x_{J_r}\vert\tilde f_2$, and so our $m_r$ has the correct form, i.e. $(m_r)_{J_r}=I_{d_r}$ (hence $\det(m_r)_{J_r}=1$), and $(m_r')_{J_r'}=I_{d_r-d_a}$. Since $(x_{J_{a+1}'}\cdots x_{J_r'})\vert f_2$,  we have $(x_{J_{a+1}}\cdots x_{J_r})\vert\tilde f_2$, hence $(m_r\cdots m_i)_{J_i}=I_{d_i}$. By induction on $i$, it is easy to see that our $m_r,..,m_i$ above have the correct form, so that
\begin{equation}\label{mr to mi}
m_r\cdots m_i=\left[\begin{matrix}m_r'\cdots m_i'& *\cr O& I_{d_a}\end{matrix}\right]
\end{equation}
and that $(m_r'\cdots m_i')_{J_i'}=I_{d_i-d_a}$ for $i=a+1..,r$. This shows that $(m_r',..,m_{a+1}')$ actually lies in the special section of $M_2(f_2)\ra X_2(f_2)$, as asserted.


{(b)} Since $x_J\vert f$, we have $(m_r\cdots m_a)_J=I_{d_a}$. From \eqref{mr to mi}, it follows that $(m_a)_J=I_{d_a}$. Since $x_{J_a}|f_1$, hence $x_{J_a}|f$, it follows that $(m_r\cdots m_a)_{J_a}$ is a nonsingular matrix $D\in GL_{d_a}$. Thus $m_a$ has the correct form as asserted, and $(m_a')_{J_a}=I_{d_a}$. This also shows that $(m_r\cdots m_a)_{1,2,..,n-d_a}=m_a'D$ has rank $d_a$, hence
$$
0\neq\bar f_1(m_r\cdots m_a,m_{a-1},..,m_1)=f_1(m_a'D,m_{a-1},..,m_1).
$$
Since $f_1$ is $GL_{d_a}$-equivariant, this is equivalent to
$$0\neq f_1(m_a',D^{-1}m_{a-1},m_{a-2},..,m_1).$$
This implies that
$$(m_a',m_{a-1}',..,m_1')=(m_a',D^{-1}m_{a-1},m_{a-2},..,m_1)$$
lies in the special section of $M_1(f_1)\ra X_1(f_1)$, as asserted.

This completes the proof.
\end{proof}

We now use the special section $m:X(f)\ra M(f)$ described in the preceding lemma to define a map
\begin{align}\label{isom-rstepcase5}
&X(f)\ra X_1(f_1)\times X_2(f_2)\times GL_{d_a}\cr
&o\mapsto [m_a'(o),..,m_1'(o)],[m_r'(o),..,m_{a+1}'(o)],D(o).
\end{align}
We will prove that this is an isomorphism. We will need the following elementary lemma.

\begin{lem}\label{matrixlemma}
Let $m_2'$ be a $(n-d_1)\times(d-d_1)$ matrix, and $A_1,A_2$ be $(d-d_1)\times d_1$ and $(n-d_1)\times d_1$ matrices. Put
$$
m_2=\left[\begin{matrix}m_2'&A_2\cr O& I_{d_1}\end{matrix}\right],~~m_1=\left[\begin{matrix}A_1\cr I_{d_1}\end{matrix}\right]
$$
and assume that $J'\subset(1,..,n-d_1)$, $|J'|=d_2-a_1$, and that the $J=J'\cup(n-d_1+1,..,n)$-block of $m_2$ is $I_d$ (which is equivalent to that $(A_2)_{J'}=O$ and $(m_2')_{J'}=I_{d-d_1}$). Then $A_1,A_2$ can be uniquely expressed as polynomial functions in terms of $m_2'$ and $m_2m_1$.
\end{lem}

\begin{lem}
The map \eqref{isom-rstepcase5} is an isomorphism.
\end{lem}

\begin{proof}
We will explicitly construct the inverse of \eqref{isom-rstepcase5}. It is enough to show that given a point
$m':=((m_a',..,m_1'),(m_r',..,m_{a+1}'),D)$ in the special section of the bundle $M_1(f_1)\times M_2(f_2)\times GL_{d_a}\ra X_1(f_1)\times X_2(f_2)\times GL_{d_a}$, the relations \eqref{relations-rstepcase5} {\it uniquely} determine a point $m=(m_r,..,m_1)\in M$, expressible polynomially in terms of $m'$. In fact, it is enough to show that the $A_a,..,A_r$ can be so-expressed. Note that the relations \eqref{relations-rstepcase5} ensures that $m$ lies in the special section of the bundle $M(f)\ra X(f)$.

By \eqref{relations-rstepcase5}, we have for $i=1,..,a+1$,
\begin{equation*}
m_r\cdots m_i=\left[\begin{matrix}m_r'\cdots m_i' & m_r'\cdots m_{i+1}'A_i+\cdots+m_r'A_{r-1}+A_r\cr
O& I_{d_a}\end{matrix}\right].
\end{equation*}
Since $m_a=\left[\begin{matrix}A_a\cr I_{d_a}\end{matrix}\right]$, Lemma \ref{matrixlemma} implies that $A_a$ and $m_r'\cdots m_{a+2}'A_{a+1}+\cdots+m_r'A_{r-1}+A_r$ can be uniquely expressed polynomially in terms of $m'$. It follows that the right hand block of $m_r\cdots m_{a+1}$:
\begin{align*}
&(m_r\cdots m_{a+1})_R=\left[\begin{matrix}m_r'\cdots m_{a+2}'A_{a+1}+\cdots+m_r'A_{r-1}+A_r\cr I_{d_a}\end{matrix}\right]\cr
&=\left[\begin{matrix}m_r'\cdots m_{a+2}' & m_r'\cdots m_{a+1}'A_{a+2}+\cdots+m_r'A_{r-1}+A_r\cr
O& I_{d_a}\end{matrix}\right]\left[\begin{matrix}A_{a+1}\cr I_{d_a}\end{matrix}\right]\cr
&=m_r\cdots m_{a+2}\left[\begin{matrix}A_{a+1}\cr I_{d_a}\end{matrix}\right]
\end{align*}
can be so-expressed. By Lemma \ref{matrixlemma} again, the right hand block of $m_r\cdots m_{a+2}$ and $A_{a+1}$ can also be so-expressed. Continuing this way, we see that $A_a,..,A_r$ all can be so-expressed. This completes the proof.
\end{proof}

The lemma and Theorem \ref{holo-rank-thm} imply

\begin{prop}\label{rstepcase5}
For $d_a+d_{a+1}>n>2d_a$ with $r>a>1$, if any $s$-step flag variety for $s<r$ admits a rank 1 point, then $X=F(d_1,..,d_r,n)$ admits one as well.
\end{prop}

\noi{\bf Case 6.} Assume $n=2d_a$ with $r>a>1$. Consider
\begin{align*}
X_1&=F(d_1,..,d_a)\equiv F(d_1,..,d_a,d_a)\into F(d_1,..,d_a,n),~~~(E^\bullet)\mapsto(0_{n-d_a}\oplus E^\bullet)\cr
X_2&=F(d_{a+1}-d_a,..,d_r-d_a,n-d_a)\into F(d_{a+1},..,d_r,n),~~~(E^\bullet)\mapsto(E^\bullet\oplus\C^{d_a})
\end{align*}
Here we view $\C^n=\C^{n-d_a}\oplus\C^{d_a}$. Let $M_1,M_2,M$ be the Stiefel bundles over $X_1,X_2,X$ respectively. Let $f_1,f_2$ be rank 1 points of $X_1,X_2$ respectively with the hyperplane properties
\begin{equation}\label{f1f2-rstepcase6}
(x_{J_1}\cdots x_{J_a})\vert f_1,~~~(x_{J_{a+1}'}\cdots x_{J_r'})\vert f_2
\end{equation}
for some $J_i\subset(1,2,..,n-d_a)$ with $\vert J_i\vert=d_i$ ($i=1,...,a-1$) and $J_{a-1}=(d_a-d_{a-1}+1,..,d_a)$, and for some $J_i'\subset(1,2,..,n-d_a)$ with $\vert J_i'\vert=d_i-d_a$ ($i=a+1,..,r$) and $J_r'=(n-d_r+1,..,n-d_a)$. Put $J:=(n-d_a+1,..,n)$, $J_i:=J_i'\cup J$, $i=a+1,..,n$, and
\begin{equation}\label{f-rstepcase6}
f:=\bar f_1\cdot\tilde f_2\cdot(x_J)^{d_{a+1}-d_{a-1}-1}(x_{1,..,d_a}).
\end{equation}
Then $f$ has the hyperplane property
\begin{equation}\label{hyperplane-rstepcase6}
(x_{J_1}\cdots\widehat{x_{J_a}}\cdots x_{J_r}x_J)\vert f.
\end{equation}

\begin{lem}
The special section $m=(m_r,..,m_1)$ (cf. Lemma \ref{special-section}) of $M(f)\ra X(f)$ has the following form:
\begin{align}\label{relations-rstepcase6}
m_i&=\left[\begin{matrix}m_i'& A_i\cr O& I_{d_a}\end{matrix}\right], ~~~i=a+1,..,r\cr
m_a&=\left[\begin{matrix}A_a\cr I_{d_a}\end{matrix}\right]\cr
m_r\cdots m_a&=\left[\begin{matrix}D\cr I_{d_a}\end{matrix}\right]\cr
m_{a-1}&=D^{-1}m_{a-1}'\cr
m_i&=m_i',~~~i=1,..,a-2
\end{align}
where $D$ is a $GL_{d_a}$-valued function, $A_a,..,A_r$ are matrix valued functions, and $(m_{a-1}',..,m_1')$, $(m_r',..,m_{a+1}')$ are matrix valued functions taking values in the special sections of the $M_1(f_1)\ra X_1(f_1)$, $M_2(f_2)\ra X_2(f_2)$ respectively.
\end{lem}

The proof is a degenerate version of the lemma in Case 5 (with $m_a'$ missing but with $(1,..,d_a)$ play the role of $J_a$), and will be omitted. The lemma and Theorem \ref{holo-rank-thm} imply

\begin{prop} \label{rstepcase6}
For $n=2d_a$ with $r>a>1$, if any $s$-step flag variety for $s<r$  admits a rank 1 point $f$ with the hyperplane property, then $X=F(d_1,..,d_r,n)$ admits one as well.
\end{prop}

\noi{\bf Case 7.} Assume $n=d_{a-1}+d_a$ with $r>a>1$.  If $a=2$ then it is Case 4, so we can assume $a\geq3$ (and $r\geq4$). Consider
\begin{align*}
X_1&:=F(d_1,..,d_{a-2},d_{a-1})\into F(d_1,..,d_{a-2},n),~~~E^\bullet\mapsto E^\bullet\oplus 0_{n-d_{a-1}}\cr
X_2&:=F(d_{a-1},..,d_r,n).
\end{align*}
Here we view $\C^n=\C^{d_{a-1}}\oplus\C^{n-d_{a-1}}$. Let $M_1,M_2,M$ be the Stiefel bundles over $X_1,X_2,X$ respectively. Let $f_1,f_2$ be rank 1 points of $X_1,X_2$ respectively with the hyperplane properties
\begin{equation}\label{f1f2-rstepcase7}
(x_{J_1}\cdots x_{J_{a-2}})\vert f_1,~~~((x_{J_{a-1}})^{d_{a-1}}x_{J_a}\cdots x_{J_r})\vert f_2
\end{equation}
for some $J_i\subset(1,2,..,d_{a-1})$ with $\vert J_i\vert=d_i$ ($i=1,...,a-2$), and for some $J_i\subset(1,2,..,n)$ with $\vert J_i\vert=d_i$ ($i=a-1,..,r$) and $J_{a-1}=(1,..,d_{a-1})$. Note that such an $f_2$ exists by Proposition \ref{rstepcase4} in Case 4, if any $s$-step flag variety for $s<r$ admits a rank 1 point with the hyperplane property.

Put
\begin{equation}\label{f-rstepcase7}
f=\bar f_1\cdot\bar f_2\cdot(x_{J_{a-1}})^{-d_{a-2}}\in\Gamma(X,\omega_X^{-1}).
\end{equation}
Then $f$ has the hyperplane property
\begin{equation}\label{hyperplane-rstepcase6}
(x_{J_1}\cdots x_{J_r})\vert f.
\end{equation}

\begin{lem}
We have an $H=GL_{d_r}\times\cdots\times GL_{d_1}$ equivariant isomorphism
\begin{align*}
&M_1(f_1)\times M_2(f_2)\ra M(f)\cr
&(m_{a-2}',..,m_1'),(m_r',..,m_{a-1}')\mapsto m=(m_r',..,m_{a-1}',D^{-1}m_{a-2}',m_{a-3}',..,m_1')
\end{align*}
where $D$ is the $J_{a-1}$-block of $m_r'\cdots m_{a-1}'$. Hence the map descends to an isomorphism
$$
X_1(f_1)\times X_2(f_2)\ra X(f).
$$
\end{lem}

The proof is almost identical to the lemmas in Cases 1 and 3, and will be omitted. The lemma and Theorem \ref{holo-rank-thm} imply

\begin{prop}\label{rstepcase7}
For $d_{a-1}+d_a=n$ with $r>a>1$,  if any $s$-step flag variety for $s<r$ admits a rank 1 point with the hyperplane property, then $X=F(d_1,..,d_r,n)$ admits one as well.
\end{prop}

Now combining the propositions in all Cases 1-7 yields a complete recursive procedure for constructing a rank 1 point with the hyperplane property for any $r$-step flag variety, proving Corollary \ref{rank1point}.

\begin{ex}
Consider $X=F(1,2,3,5)$, which belongs in Case 3. Let $X_1=F(1,2)$ and take $f_1=x_1x_2$. Let $X_2=F(2,3,5)$, which belongs in Case 2, and we can take $f_2=(x_{12})^3(x_{345})^3$ as a rank 1 point of $X_2$, by Proposition \ref{2stepcase2}. Therefore,
$$
f=x_1x_2(x_{12})^3(x_{345})^3(x_{12})^{-1}
$$
is rank 1 point of $X$ according to the construction in Case 3.
\end{ex}

\begin{ex}
Consider the flag variety of $SL_5$, $X=F(1,2,3,4,5)$, which belongs in Case 7 with $a=3$. Let $X_1=F(1,2)$ and take $f_1=x_1x_2$. Let $X_2=F(2,3,4,5)\simeq F(1,2,3,5)$, which is the preceding example. Applying this isomorphism to the rank 1 point there, we get $f_2=x_{2345}x_{1345}(x_{345})^2(x_{12})^3$ as a rank 1 point of $X_2$. Therefore,
$$
f=x_1x_2x_{2345}x_{1345}(x_{345})^2(x_{12})^3(x_{12})^{-1}
$$
is a rank 1 point of $X$ according to the construction in Case 7.
\end{ex}

\appendix

\section{Theory of $D$-modules}\label{Appendix}
We recall the theory of \emph{algebraic} $D$-modules. A standard
reference is \cite{Borel}. 

Let $X$ be an algebraic variety over $k$ of characteristics zero. Let
$\Hol(D_X)$ be the category of holonomic (left) $D$-modules on $X$. Its
bounded derived category is denoted by $D_h^b(X)$.

Let $f:X\to Y$ be a morphism, there are the following pairs of
adjoint (derived) functors (following the notation of Borel's book)
\[f^+:D^b_h(Y)\rightleftharpoons D_h^b(X): f_+, \quad\quad f_!:D_h^b(X)\rightleftharpoons D_h^b(Y): f^!.\]
Recall the definition of $f_+$ in the following cases (assuming $X$
and $Y$ are smooth): in the case, there is an $f^{-1}D_Y\times
D_X$-bimodule $D_{Y\leftarrow X}$ on $X$, and
\[f_+(M)=Rf_*(D_{Y\leftarrow X}\otimes^L M).\]
Without mentioning the exact definition of this bimodule
$D_{Y\leftarrow X}$, we concentrate on the following special cases. Let $d_{X,Y}=\dim X-\dim Y$.

(i) $f:X\to Y$ is smooth. Then $f_+$ (up to shift) is the usual
construction of the Gauss-Manin connection. I.e.
\[f_+(M)=Rf_*(M\otimes \Omega_{X/Y}^\bullet[d_{X,Y}]).\]
In particular, $H^i f_+\calO_X$ is the $D$-module on $Y$ formed by
the $(i+d_{X,Y})$th relative De Rham cohomology. In
particular, if $f$ is an open embedding, then $f_+(M)=Rf_*M$ as
quasi-coherent sheaves on $Y$. Observe that under the this normalization of the cohomological degrees, $H^0f_+\calO_X$ is the usual ``middle dimension" cohomology of the family $f:X\to Y$.

\begin{ex}\label{star extension}
A particular example: $j: \bG_m=\Spec k[x,x^{-1}]\to \bA^1=\Spec
k[x]$ the open embedding. Then $j_+\calO_{\bG_m}$ as a $D$-module on
$\bA^1$ is isomorphic to $k[x,\partial_x]/(x\partial_x+1)$.
\end{ex}

(ii) $f:X\to Y$ is a closed embedding given by the ideal $\calI$.
Then
\[f_+(M)=f_*(D_Y/D_Y \calI\otimes \omega_{X/Y}\otimes M)\]
where $\omega_{X/Y}$ is the relative canonical sheaf
$\omega_{X/Y}=\omega_X\otimes(\omega_Y^{-1}|_X)$.

\begin{ex}\label{delta sheaf}
A particular example: let $Y$ be $\bA^n=\Spec k[x_1,\ldots,x_n]$ and
$i:X\to Y$ be the inclusion of the vector space given by
$x_1=\cdots=x_r=0$. Then $x_{r+1},\ldots,x_n$ form a coordinate
system on $X$. Let
$M=\calO_X=D_X/D_X(\partial_{r+1},\ldots,\partial_n)$. Then
\[i_+M= D_Y/D_Y(x_1,\ldots,x_r,\partial_{r+1},\ldots,\partial_n),\]
called the delta sheaf supported on $X$, denoted by $\delta_X$.
\end{ex}

Observe that there is the following exact sequence of
$D_{\bA^1}$-modules
\begin{equation}\label{fundamental}
0\to\calO_{\bA^1}\to j_+\calO_{\bG_m}\to \delta_{\{0\}}\to 0.
\end{equation}

Next, we recall the definition of $f^!$. There is a $D_X\times
f^{-1}D_Y$-bimodule $D_{X\to Y}$ on $X$, and by definition
\[f^!(M)= D_{X\to Y}\otimes^L_{f^{-1}D_Y}f^{-1}M[d_{X,Y}].\]
As quasi-coherent $\calO_X$-modules,
$$
f^!(M)=Lf^*M[d_{X,Y}].
$$
Again, let us mention the following special cases.

(i) $f:X\to Y$ is smooth. In this case, $f^![-d_{X,Y}]$ is
exact, and as quasi-coherent sheaves, $f^![-d_{X,Y}](M)=f^*M$.
In particular, if $f$ is an open embedding, then $f^!M=M|_X$.

(ii) $f:X\to Y$ is a closed embedding, given by the ideal sheaf
$\calI$. In this case
\[H^0f^!(M)\otimes \omega_{X/Y}=\{m\in M\mid xm=0 \mbox{ for any } x\in \calI\}.\]

The following distinguished triangle generalizes
\eqref{fundamental}: Let $i:X\to Y$ be a closed embedding and
$j:U\to Y$ be the complement.
\begin{equation}\label{triangle}
i_+i^!M\to M\to j_+j^!M\to.
\end{equation}

Indeed, in the case $Y=\bA^1$ and $X=\bG_m$, $M=\calO_{\bA^1}$, we
recover \eqref{fundamental}.

The following theorem (Kashiwara's lemma) is of fundamental
importance,
\begin{thm}Let $i:X\to Y$ be a closed embedding.

(i) If $M$ is a $D_Y$-module, set-theoretically supported on $X$.
Then $H^ii^!M=0$ for $i>0$.

(ii) Let $D_Y\Mod_X$ be the category of $D_Y$-modules,
set-theoretically supported on $X$, and $D_X\Mod$ be the category of
$D_X$-modules. Then there is an equivalence of categories
\[i_+: D_X\Mod\rightleftharpoons D_Y\Mod_X:  H^0i^!.\]
\end{thm}

In the sequel, we will make use of the following notation: let $i: X\to Y$ be a locally closed embedding. If $M$ is a D-module on $Y$, set-theoretically supported on $\bar{X}$, then $H^0 i^!M$ will be denoted by $M|_{X}$.

This finishes the discussion of the functors $f_+, f^!$. Then $f_!$
is defined to be the left adjoint of $f^!$ and $f^+$ is defined to
be the left adjoint of $f_+$. Recall that there is the duality
functor $\bD_X: D_h^b(X)\to D_h^b(X)$. We can also express
$f^+=\bD_Xf^!\bD_Y$ and $f_!=\bD_Yf_+\bD_X$. It is known that

(i) If $f:X\to Y$ is a closed embedding (or more generally if $f$ is
proper), $f_!=f_+$.

(ii) If $f:X\to Y$ is an open embedding, $f^!=f^+$.

\begin{rmk}The definitions of $f_+,f^!$ do not require the
holonomicity, and therefore they are defined on the whole category
of (not necessarily holonomic) $D$-modules. However, as functors on
the whole category of $D$-modules, they do not admit adjoint
functors and therefore $f_!, f^+$ are not defined in general.
\end{rmk}


\begin{ex}Let $j:\bG_m\to \bA^1$ as before. One can show that
$j_!\calO_{\bG_m}\simeq k[x,\partial]/x\partial$.
\end{ex}

The dual version of \eqref{fundamental} is
\begin{equation}\label{shrik extension}0\to \delta_{\{0\}}\to j_!\calO_{\bG_m}\to
\calO_{\bA^1}\to 0,\end{equation} and the dual version of
\eqref{triangle} is
\begin{equation}\label{triangleII}
j_!j^!M\to M\to i_+i^+M\to .
\end{equation}

\medskip

Now let $k=\bC$. Let $D^b_{rh}(X)$ be the bounded derived category of
holonomic $D$-modules with regular singularities, and let
$D^b_c(X^{an})$ be the bounded derived category of constructible
sheaves on $X^{an}$ (we denote $X$ equipped with the classical
topology by $X^{an}$). Then Riemann-Hilbert correspondence is an
equivalence
\[\on{RH}: D^b_{rh}(X)\simeq D^b_c(X), \quad\quad \on{RH}(M)=\omega_{X^{an}}\otimes^LM^{an}=\Omega_{X^{an}}^\bullet\otimes M^{an}[\dim X],\]
where $\omega_{X^{an}}$ is the canonical sheaf on $X^{an}$, regarded
as a right $D$-module via Lie derivative, and the derived tensor
product is over $D_{X^{an}}$. This correspondence is compatible with
the six operation functors. In particular,
\[\on{RH}f_+\simeq f_*\on{RH},\ \ \on{RH}f_!\simeq f_!\on{RH},\ \ \on{RH}f^!\simeq f^!\on{RH},\ \ \on{RH}f^+\simeq f^*\on{RH}.\]
If $M$ is a plain $D$-module, then $\on{RH}(M)$ is a perverse sheaf
on $X^{an}$. While the above equivalence is covariant, sometimes one
also consider the contravariant version
\[\on{Sol}: D^b_{rh}(X)\simeq D^b_c(X)^{op}, \quad\quad \on{Sol}(M)=R\Hom_{D_{X^{an}}}(M^{an},\calO_{X^{an}}).\]
The relation between $\on{Sol}$ and $\on{RH}$ is
$\on{RH}=\on{Sol}\bD_X[\dim X]$.

\begin{rmk}\label{classical sol}
Let $M$ be a D-module on $X$. In the paper we also talk about the solution sheaf of $M$, by which we mean the classical (non-derived) solutions of $M$, and is defined as
\[{^{cl}}\on{Sol}(M)=\Hom_{D_{X^{an}}}(M^{an},\calO_{X^{an}}).\]
This is a plain sheaf on $X^{an}$.
\end{rmk}

\medskip

Next, we discuss background materials on equivariant D-modules, most which can be found in \cite{Borel}\cite{Hotta}.
Let $G$ be a connected algebraic group and $\frakg=\Lie G$. Let us regard $\frakg$ as right invariant vector fields on $G$, and
for a Lie algebra homomorphism $\chi:\frakg\to k$, we define a character D-module on $G$ by
\begin{equation}\label{char dmod}
\calL_\chi= D_G/D_G(\xi+\chi(\xi), \xi\in \frakg).
\end{equation}
This is a rank one local system on $G$. In particular, it is holonomic. It is called a character sheaf because if we denote by $\on{mult}:G\times G\to G$ the multiplication map of $G$, then there is a canonical isomorphism $\on{mult}^!\calL_\chi\simeq \calL_\chi\boxtimes\calL_\chi[\dim G]$ satisfying the cocycle condition under the further $!$-pullback to $G\times G\times G$.

Let $Z$ be a $G$-variety and $\on{act}:G\times Z\to Z$ be the action map. A $(G,\chi)$-equivariant, or a $G$-monodromic against $\chi$, D-module on $Z$ is a D-module on $Z$ together with an isomorphism
\[\theta: \on{act}^!M\simeq \calL_{\chi}\boxtimes M[\dim G],\]
satisfying the usual cocycle condition under the further $!$-pullback to $G\times G\times Z$.

The following lemma is well-known, which can be proved as in \cite[Theorem 12.11]{Borel}. See also \cite[\S II.5]{Hotta}.
\begin{lem}\label{mono:rh}
Assume that there are only finitely many orbits under the action of $G$ on $Z$, then any $(G,\chi)$-equivaraint D-module is holonomic. In addition, if $\calL_\chi$ is regular singular, then any $(G,\chi)$-equivariant D-module is regular singular.
\end{lem}

We will need the following lemma. Let $U\frakg$ be the universal enveloping algebra of $\frakg$. Then $\chi$ defines a one-dimensional $U\frakg$-module, denoted by $k_\chi$. Note that if $Z$ is a $G$-variety, we have the corresponding infinitesimal action $da:\frakg\to T_Z$, which extends to $U\frakg\to D_Z$.
\begin{lem}\label{mono:eq}
The D-module
\[D_{Z,\chi}=D_Z/D_Z(da(\xi)+\chi(\xi), \xi\in \frakg)=(D_Z\otimes k_{\chi})\otimes_{U\frakg}k\]
is a natural $(G,\chi)$-equivariant D-module on $Z$.

More generally, note that $D_Z$ is naturally $G$-equivariant as $\calO$-modules, i.e., there is an isomorphism of $\calO$-modules
$\theta:\on{act}^* D_Z\simeq p_Z^* D_Z$
satisfying the cocycle condition.
Let $I\subset D_Z$ be a $G$-invariant left ideal, then
$$D_Z/I+D_Z(da(\xi)+\chi(\xi), \xi\in \frakg)$$ is $(G,\chi)$-equivariant.
\end{lem}
Note that in the above lemma, we do not need to assume that $G$ acts on $Z$ with finitely many orbits.  See \cite[\S II.3]{Hotta}.

\quash{
\begin{proof}
Observe that up to shift, the
underlying $\calO$-module for $\on{act}^!D_Z$ (resp. $p_X^!D_Z$) is
$\on{act}^*D_Z$ (rep. $p_Z^*D_Z$). However, $\theta$ is not an
isomorphism of $D$-modules. More precisely, let $D$ be a local
section of $D_Z$, which pulls back to a local section of
$\on{act}^*D_Z$, still denoted by $D$. Let $\xi\in \frakg$, regarded
as a right invariant vector field on $G$, and therefore a vector
field on $G\times Z$, then
\[\xi\theta(D)=\theta(da(\xi) D)+\theta(D da(\xi)).\]
The equivariance structure $\theta$ on $D_Z$ induces an
equivariance structure on $M$. Therefore, if $m\in M$ is a
local section, lifted to a local section $\tilde{m}$ of $D_Z$,
\[\xi\theta(m)=\theta(da(\xi)\tilde{m}-\tilde{m}da(\xi))=\theta((da(\xi)+\chi(\xi))m).\]
Or
\[(\xi-\chi(\xi))\theta(m)=\theta(da(\xi)m).\]
\end{proof}}

Note that if $i:H\to G$ is a connected closed subgroup, $i^!\calL_{\chi}[-\dim H]=D_{H}/D_{H}(\xi+\chi(\xi), \xi\in \frakh)=\calL_{\chi|_\frakh}$. We have the following simple observation.
\begin{lem}\label{mono:st}
Let $Z=G/H$ be a homogeneous $G$-variety. Let $\chi:\frakg\to k$ be a Lie algebra homomorphism and $\calL_\chi$ be the rank character D-module on $G$ as in \eqref{char dmod}.  Then if $\calL_{\chi|\frakh}\neq \calO_{H^\circ}$, where $H^\circ$ is the neutral connected component of $H$, there is no D-module on $Z$, equivariant with respect to $G$ against $\chi$.\end{lem}
\begin{proof}Let $M$ be a non-zero $(G,\chi)$-equivariant D-modules on $Z$. Let $i:H^\circ\to G$ be the inclusion, and $i_e:eH\to Z$ be the inclusion of the identity coset.
Consider the diagram
\[\begin{CD}
H^\circ\times eH@>>>eH\\
@Vi\times i_eVV@VVi_eV\\
G\times Z@>>>Z.
\end{CD}\]
Then $i^!\calL_\chi\otimes i_{e}^!M=(i\times i_e)^!\on{act}^!M=\calO_H\otimes i_e^!M[\dim Z]$. Therefore, $\calL_{\chi|_\frakh}=\calO_H$.
\end{proof}

\begin{ex}\label{clean}
Let $\lambda\in k^\times$, and let $\calL_\lambda$ be the $D$-module
on $\bG_m$ given by $x\partial+\lambda$. I.e. $\calL_\la$ is the
local system on $\bG_m$ with monodromy $\exp(-2\pi\sqrt{-1}\la)$
(via the Riemann-Hilbert correspondence if $k=\bC$). This is a character D-module on $G$ with $\chi(x\partial)=\la$.
If $\la\in\bZ$, then
$\calL_\la\simeq\calO_{\bG_m}$.
Let
$j:\bG_m\to\bA^1$ be the open embedding. Then both $j_+\calL_\la$ and $j_!\calL_\la$ are $(G,\la)$-equivaraint D-modules on $\bG_m$.
If $\la$ is not an integer, then
$j_!\calL_{\la}\simeq j_+\calL_\la$. In this case, this $D$-module
is irreducible on $\bA^1$.
\end{ex}

\medskip

Our last topic is the Fourier transform.
Let $``e^x"$ be the character $D$-module on $\bA^1$ defined by $\partial-1$.
Let $V$ be a vector space and $V^\vee$ be its dual. We have the natural pairing
\[m:V\times V^\vee\to \bA^1.\]
The pullback of $e^x$ along $m$ is still denoted by $e^x$, regarded as a plain D-module on $V\times V^\vee$. Let $p_V,
p_{V^\vee}$ be the projections of $V\times V^\vee$ to the two
factors. The Fourier transform is defined as
\[\four(M)=p_{V^\vee,+}(p_V^!(M)\otimes e^x),\]

Fourier transform $\four$ is an exact
functor, and can be described in the following simple way. Let $M$
be a $D$-module on $V$, and therefore is identified with a module
over the Weyl algebra
$k[a_1,\ldots,a_n,\partial_{a_1},\ldots,\partial_{a_n}]$. Then
$\four(M)$ as a vector space is identified with $M$, and the
$D$-module structure is given by $a_i^*m=\partial_{a_i}m$ and
$\partial_{a_i^*}=-a_im$. In other words, if we denote the ring
homomorphism
\begin{equation}\label{appen:ring hom}
\widehat{}: D_V\to D_{V^\vee}, \quad \widehat{a_i}=-\partial_{a_i^*},\ \widehat{\partial_{a_i}}=a_i^*,
\end{equation}
then $\four(M)=D_{V^\vee}\otimes_{D_V}M$.  See \cite[p85]{Bry}.

\begin{ex}Let $W\subset V$ be a vector subspace, and $W^\perp$ be
the orthogonal complement of $W$ in $V^\vee$. Then
$\four(\delta_W)=\delta_{W^\perp}$.
\end{ex}

\begin{ex}\label{four:change space}
More generally, let $i:W\subset V$ be a vector subspace, and $0\to W^\perp\to V^\vee\stackrel{p}{\to} W^\vee\to 0$ be the dual sequence.
Let $M$ be a D-module on $W$. Then $$\four(i_+M)= p^!\four(M)[\dim W-\dim V].$$
\end{ex}

\begin{ex}\label{trans}
Let $V=\bA^1$ and we identify $V^\vee=\bA^1$ via the natural
multiplication $\bA^1\times\bA^1\to\bA^1$. Then under the Fourier
transform, the exact sequence \eqref{fundamental} becomes
\eqref{shrik extension}.
\end{ex}

\begin{ex}\label{four:mono}
Recall the character D-module $\calL_\la$ on $\bG_m$. Let $j:\bG_m\to \bA^1$ be the open immersion. Then
\[\four(j_+\calL_\la)=j_!\calL_{-\la+1}.\]
\end{ex}

\quash{
\begin{ex}\label{mix}
Let $i:L=\bA^1\subset V$, and $j_U:U=V^\vee\setminus L^\perp\to
V^\vee$. Let $j:\bG_m\to\bA^1$. Then
\[\four(i_+j_+\calO_{\bG_m})=j_{U,!}\calO_U.\]
Dually,
\[\four(i_!j_!\calO_{\bG_m})=j_{U,+}\calO_U.\]
More generally, let us write $L\subset V$ given by
$a_2=\cdots=a_n=0$, and $L^\perp\subset V^\vee$ given by $a_1^*=0$.
So that $U\simeq \bG_m\times\bA^{n-1}$. Let $\la\in k^\times$. Then
we can identify
\[\four(i_+j_+\calO_{-\la})=j_{U,!}(\calO_{\la+1}\boxtimes \calO_{\bA^{n-1}}).\]
\end{ex}}

Fourier transform preserves holonomicity. If $M$ is holonomic, the
we can also write
\[\four(M)=p_{V^\vee,!}(p_V^*(M)\otimes e^x).\]
However, Fourier transform does not necessarily preserves the
regular singularity. For example, the Fourier transform of the delta
sheaf on $\bA^1$ supported at $1\in\bA^1(k)$ is $e^x$. However,
under certain circumstance, one can show that $\four(M)$ is regular
singular. Let $\bG_m$ act on $V$ via homotheties, i.e.
$\on{mult}:\bG_m\times V\to V, \on{mult}(a,v)=av$. Let $\la: \Lie\bG_m\to k$ be a map. Recall the notion of $(\bG_m,\la)$-equivariant D-modules. We say a holonomic
$D$-module on $V$ to be $\bG_m$ monodromic if each of its
irreducible constitutes is $(\bG_m,\la)$-equivariant for some $\la$. Observe that $e^x$ is not $\bG_m$-monodromic.

Let $D_{rh,m}^b(V)$ be the full subcategory of $D_{rh}^b(V)$ whose
cohomology sheaves are regular holonomic and $\bG_m$-monodromic.

\begin{lem}\label{rh}
The Fourier transform restricts to an equivalence
\[\four: D_{rh,m}^b(V)\simeq D_{rh,m}^b(V^\vee).\]
\end{lem}
\begin{proof}\cite[Theorems 7.4, 7.24]{Bry}.
\end{proof}

Fourier transform can be generalized to family versions. Let $X$ be
a base variety, and $\bV$ a vector bundle over $X$, $\bV^\vee$ the dual
bundle, so there is
\[m:\bV\times_X\bV^\vee\to \bA^1.\]
Then one can define
\[\four_X(M)=p_{\bV^\vee,+}(p_{\bV}^!(M)\otimes e^x).\]
Note that the family version of Example \ref{four:change space} still holds. More precisely, let $i:\bW\subset\bV$ be a subbundle on $p:\bV^\vee\to\bW^\vee$ be the dual map. Then
\begin{equation}\label{fcs}
\four_X(i_+M)=p^!\four_X(M)[\on{rk}\bW-\on{rk}\bV].
\end{equation}

Note that family version of Fourier transforms commute with base change. Namely, if $f:Y\to X$ is a map. Then $\four_Y(f^!M)=f^!\four_X(M)$. 

Let us consider the family version of Example \ref{four:mono}. So we assume that $\bV=\bL$ is a line bundle, on which $\bG_m$ acts by homotheties. Let $\mathring{\bL}=\bL-X$, where $X$ is regarded as the zero section of $\bL$. Let $\bL^\vee$ be the dual vector bundle of $\bL$ and $\mathring{\bL}^\vee$ is defined similarly.
Let $M$ be a $(\bG_m,\la)$-equivariant D-module on $\mathring{\bL}$.

The following
lemma is useful.
\begin{lem}\label{key}
Let $X$ be proper and $\bV=X\times V$ be the trivial bundle over
$X$. Let $\pi:X\times V\to V$ and $\pi^\vee:X\times V^\vee\to
V^\vee$ be the projections. Then
\[\four\circ \pi_!\simeq \pi^\vee_!\circ\four_X.\]
\end{lem}
\begin{proof}
This follows from the base change theorem for D-modules (cf. \cite[VI, \S 8]{Borel}). Namely, as $X$ is proper, $\pi_+=\pi_!$, etc. We have the following commutative diagrams with both squares Cartesian.
\[\xymatrix{
&X\times V\times V^\vee\ar^{\pi^{V,V^\vee}}[d]\ar^{p_V}[dl]\ar^{p_{V^\vee}}[dr]& \\
X\times V\ar_{\pi}[d]& V\times V^\vee\ar^{p_V}[dl]\ar^{p_{V^\vee}}[dr] & X\times V^\vee\ar^{\pi^\vee}[d]\\
V&&V^\vee.
}
\]
Then
\[\begin{array}{lcl}
\four(\pi_!(M))& =& p_{V^\vee,+}(p_V^!\pi^\vee_+(M)\otimes e^x)\\
& = & p_{V^\vee,+}(\pi^{V,V^\vee}_+p_V^!(M)\otimes e^x) \\
 & =  & p_{V^\vee,+}\pi^{V,V^\vee}_+(p_V^!(M)\otimes e^x) \\
  & = &\pi^{\vee}_+\four_X(M)
\end{array}\]
\end{proof}

\vskip.5in

\noindent\address {\SMALL A. Huang, Department of Mathematics, Harvard University, Cambridge MA 02138. \\
\vskip-.45in  anhuang@math.harvard.edu.}
\vskip-.15in

\noindent\address {\SMALL B.H. Lian, Department of Mathematics, Brandeis University, Waltham MA 02454.\\
\vskip-.45in   lian@brandeis.edu.}
\vskip-.15in

\noindent\address  {\SMALL X. Zhu, Department of Mathematics, California Institute of Technology, Pasadena CA 91125 \\
\vskip-.45in xzhu@caltech.edu.}


\begin{thebibliography}{10}

\bibitem{Adolphson} A. Adolphson, \emph{Hypergeometric Functions and Rings Generated by Monomials}, Duke Math. J. Vol 73, No. 2 (1994) 269-290.

\bibitem{BGG} J. Bernstein, I. M. Gelfand, S. I. Gelfand, \emph{A Certain Category of g-modules}, Functional
Analysis and its Applications. {\bf 10} no. 2 (1976), 1-8.

\bibitem{BGS} A. Beilinson, V. Ginzburg, W. Soergel, \emph{Koszul Duality Patterns in Representation Theory}, Journ. AMS {\bf 9} (1996), 473-527.

\bibitem{BHLSY} S. Bloch, A. Huang, B.H. Lian, V. Srinivas, and S.-T. Yau, \emph{On the Holonomic Rank Problem}, arXiv:1302.4481v1.

\bibitem{Borel} A. Borel et al, \emph{Algebraic D-modules}, Academic Press 1987.








\bibitem{Bry} J.-L. Brylinski, \emph{Transformations Canoniques, Dualit\'e projective, Th\'eorie de Lefschetz, Trans-
formations de Fourier et Sommes Trigonom\'etriques}, Ast\'erisque 140-141 (1986), 3734.

\bibitem{CL2014} J. Chen and B.H. Lian, \emph{CY Principle Bundles on Compact K\"ahler Manifolds}, to appear.

\bibitem{GKZ1990} I. Gel'fand, M. Kapranov and A. Zelevinsky, \emph{Hypergeometric Functions and Toral Manifolds},
English translation, Functional Anal. Appl. {\bf 23} (1989), 94-106.

\bibitem{Hague10} C. Hague, \emph{On the B-canonical Splittings of Flag Varieties}, J. Algebra 323 (2010), no. 6, 1758?764. arXiv:0908.4354.



\bibitem{HLY1996} S. Hosono, B.H. Lian and S-T. Yau, \emph{Maximal Degeneracy Points of GKZ Systems,} Journ. AMS Vol. 10, No. 2 (1997) 427-443.

\bibitem{HLY1994} S. Hosono, B.H. Lian and S-T. Yau, \emph{GKZ-generalized Hypergeometric Systems in Mirror Symmetry of Calabi-Yau hypersurfaces}, Commun. Math. Phys. {\bf 182} (1996), 535-577.

\bibitem{Hotta} R. Hotta, \emph{Equivariant D-modules}, arXiv:math/9805021v1.

\bibitem{Kapranov1997} M. Kapranov, \emph{Hypergeometric Functions on Reductive Groups}, Integrable Systems and Algebraic Geometry (Kobe/Kyoto, 1997), 236-281, World Sci. Publ., River Edge, NJ, 1998.

\bibitem{KLS10} A. Knutson, T. Lam and D. Speyer, \emph{Projections of Richardson Varieties}, arXiv:1008.3939v2.

\bibitem{KLS11} A. Knutson, T. Lam and D. Speyer, \emph{Positroid Varieties: Juggling and Geometry}, arXiv:1111.3660v1.

\bibitem{La} G. Laumon, \emph{Transformation de Fourier Homog\`{e}ne} Bull. Soc. math. France 131 (4), 2003, p.527-551.


\bibitem{LSY} B.H. Lian, R. Song and S.-T. Yau, \emph{Period Integrals and Tautological Systems}, Journ. EMS Vol. 15, 4 (2013) 1457-1483. arXiv:1105.2984v3


\bibitem{LY} B.H. Lian and S.-T. Yau, \emph{Period Integrals of CY and General Type Complete Intersections}, Invent. Math. Vol 191, 1 (2013) 35-89. arXiv:1105.4872v3.

\bibitem{Lusztig98} G. Lusztig, \emph{Total Positivity in Partial Flag Manifolds}, Representation Theory, AMS, Vol. 2 (1998)  70?8.

\bibitem{ravi} R. Virk, \emph{Extensions of Verma Modules}, http://math.colorado.edu/~ravi1033/pub/verma.pdf.

\bibitem{Reitsch05} K. Reitsch, \emph{Closure Relations for Totally Nonnegative Cells in $G/P$}, arXiv:0509137v2.

\end{thebibliography}
\end{document}